\documentclass[11pt]{article}
\pdfoutput=1

\usepackage{graphicx}
\usepackage{amsmath}
\usepackage{amssymb,amsfonts}
\usepackage{amsthm}
\usepackage{cite}
\usepackage{bm}
\newtheorem{thm}{Theorem}[section]
\usepackage{mathrsfs}

\usepackage{amssymb}   
\usepackage{multirow}

\usepackage[margin=1in]{geometry}

\newcommand{\norm}[1]{\left\lVert#1\right\rVert}

\newtheorem{lemma}{Lemma}
\newtheorem{theorem}{Theorem}
\newtheorem{proposition}{Proposition}
\newtheorem{corollary}{Corollary}

\theoremstyle{definition}
\newtheorem{remark}{Remark}

\usepackage{hyperref}
\usepackage[nameinlink]{cleveref}

\usepackage{cite}

\begin{document}

\title{A Superconvergent Hybridizable Discontinuous Galerkin Method for Dirichlet Boundary Control of Elliptic PDEs}

\author{Weiwei Hu%
		\thanks{Department of Mathematics, Oklahoma State
		University, Stillwater, OK (weiwei.hu@okstate.edu). W.~Hu was supported in part by a postdoctoral fellowship for the annual program on Control Theory and its Applications at the Institute for Mathematics and its Applications (IMA) at the University of Minnesota.}%
	 \and
	 	Jiguang Shen%
	 	\thanks{School  of Mathematics, University of Minnesota, MN (shenx179@umn.edu)}%
	 \and
	 	John~R.~Singler%
	 	\thanks{Department of Mathematics
	 		and Statistics, Missouri University of Science and Technology,
	 		Rolla, MO (\mbox{singlerj@mst.edu}, ywzfg4@mst.edu). J.~Singler and Y.~Zhang were supported in part by National Science Foundation grant DMS-1217122.  J.~Singler and Y.~Zhang thank the IMA for funding research visits, during which some of this work was completed.}
	 \and
	 	Yangwen Zhang%
	 	\footnotemark[3]
	 \and
		Xiaobo Zheng%
		\thanks{College of Mathematics, Sichuan University,
			Chengdu, China (zhengxiaobosc@yahoo.com). X.~Zheng thanks Missouri University of Science and Technology for hosting him as a visiting scholar; some of this work was completed during his research visit.}
}

\maketitle

\begin{abstract}
	We begin an investigation of hybridizable discontinuous Galerkin (HDG) methods for approximating the solution of Dirichlet boundary control problems governed by elliptic PDEs.  These problems can involve atypical variational formulations, and often have solutions with low regularity on polyhedral domains.  These issues can provide challenges for numerical methods and the associated numerical analysis.  We propose an HDG method for a Dirichlet boundary control problem for the Poisson equation, and obtain optimal a priori error estimates for the control. Specifically, under certain assumptions, for a 2D convex polygonal domain we show the control converges at a superlinear rate. We present 2D and 3D numerical experiments to demonstrate our theoretical results.
%
%
\end{abstract}

\section{Introduction}
We consider the following elliptic Dirichlet boundary control problem on a Lipschitz polyhedral domain $\Omega\subset \mathbb{R}^{d} $, $ d\geq 2 $, with boundary $\Gamma = \partial \Omega:$
\begin{align}
\min J(u),  \quad  J(u)=\frac{1}{2}\|y-y_{d}\|^2_{L^{2}(\Omega)}+\frac{\gamma}{2}\|u\|^2_{L^{2}(\Gamma)}, \label{cost1}
\end{align}
where $ \gamma > 0 $ and $ y $ is the solution of the Poisson equation with nonhomogeneous Dirichlet boundary conditions
\begin{align}
-\Delta y &= f \quad  \text{in} \; \Omega \label{p1},\\
y &= u \quad  \text{on}\;  \Gamma. \label{p2}
\end{align}

It is well known that the Dirichlet boundary control problem \eqref{cost1}-\eqref{p2} is equivalent to the optimality system
\begin{subequations}\label{boundary_pro}
	\begin{align}
	-\Delta y &=f \quad  \quad  \quad \; \text{in} \; \Omega \label{boundary_pro_a},\\
	y&=u \quad \quad \quad\; \text{on}\;  \Gamma, \label{boundary_pro_u}\\
	-\Delta z  &=y-y_d \quad \; \text{in} \; \Omega \label{boundary_pro_b},\\
	z&=0 \quad \quad \quad\;\; \text{on}\;  \Gamma, \label{boundary_pro_c}\\
	u& =\gamma^{-1}\frac{\partial z}{\partial \bm{n}}\quad \text{on}\;  \Gamma. \label{boundary_pro_d}
	\end{align}
\end{subequations}
where $\bm{n}$ is the unit outer normal to $\Gamma$. 

Dirichlet boundary control has many applications in fluid flow problems and other fields, and therefore the mathematical study of these control problems has become an important area of research.  Major theoretical and computational developments have been made in the recent past; see, e.g., \cite{MR1013883,MR1145711,MR1613873,MR1663656,MR1788073,MR2179484,MR3379275,MR3434480,MR3666648,MR2038338,MR2154110,MR2204633,MR3227252}. However, only in the last ten years have researchers developed thorough well-posedness, regularity, and finite element error analysis results for elliptic PDEs; see \cite{MR2272157,MR2558321,MR3070527,MR2347691,ApelMateosPfeffererRosch17} and the references therein.  One difficulty of Dirichlet boundary control problems is that the Dirichlet boundary data does not directly enter a standard variational setting for the PDE; instead, the state equation is understood in a very weak sense.  Also, solutions of the optimality system typically do not have high regularity on polyhedral domains; corners cause the normal derivative of the adjoint state $ \partial z / \partial \bm{n} $ in the optimality condition \eqref{boundary_pro} to have limited smoothness.  Solutions with limited regularity can lead to complications for numerical methods and numerical analysis.


To avoid the difficulties described above, researchers have considered other approaches including modified cost functionals \cite{MR3317816,MR1135991,MR1145711,MR3614013}, approximating the Dirichlet boundary condition with a Robin boundary condition \cite{MR1874072,MR2020866,MR2567245,MR1632548,ravindran2017finite}, and weak boundary penalization \cite{MR3641789}.


One way to approximate the solution of the original problem without penalization and also avoid the variational difficulty is to use a mixed finite element method.  Recently, Gong and Yan \cite{MR2806572} considered this approach and obtained
\begin{align*}
\norm{u-u_h}_{0,\Omega} = O(h^{1-1/s})
\end{align*}
when the control belongs to $H^{1-1/s}(\Gamma)$ and the lowest order Raviart-Thomas elements are used for the computation.


As researchers continue to investigate Dirichlet boundary control problems of increasingly complexity, it may become natural to utilize discontinuous Galerkin methods for the spatial discretization of problems involving strong convection and discontinuities.  We have performed preliminary computations using an hybridizable discontinuous Galerkin (HDG) method for a similar elliptic Dirichlet boundary control problem for the Stokes equations.  Our preliminary results for this problem indicate that the optimal control can indeed be discontinuous at the corners of the domain.  Before we continue to investigate problems of such complexity, we begin this line of research by considering an HDG method to approximate the solution of the above Dirichlet boundary control problem.

HDG methods also utilize a mixed formulation and therefore avoid the variational difficulty of the Dirichlet control problem.  Furthermore, the number of degrees of freedom for HDG methods are much less than standard mixed methods or other DG approaches.  Moreover, the RT element is a special case of the HDG method.  We provide more background about HDG methods in Section \ref{sec:HDG}.


We propose an HDG method to approximate the control in Section \ref{sec:HDG}, and in Section \ref{sec:analysis} we prove an optimal superlinear rate of convergence for the control in 2D under certain assumptions on the domain and $ y_d $.  To give a specific example, for a rectangular 2D domain and $ y_d\in H^1(\Omega) \cap L^\infty(\Omega)$, we obtain the following a priori error bounds for the state $ y $, adjoint state $ z $, their fluxes $ \bm{q} = -\nabla y $ and $ \bm{p} = -\nabla z $, and the optimal control $ u $:
%
\begin{align*}
&\norm{y-{y}_h}_{0,\Omega}=O( h^{3/2-\varepsilon} ),\quad  \;\norm{z-{z}_h}_{0,\Omega}=O( h^{3/2-\varepsilon} ),\\
&\norm{\bm{q}-\bm{q}_h}_{0,\Omega} = O( h^{1-\varepsilon} ),\quad  \quad\;\; \norm{\bm{p}-\bm{p}_h}_{0,\Omega} = O( h^{3/2-\varepsilon} ),
\end{align*}
and
\begin{align*}
&\norm{u-{u}_h}_{0,\Gamma} = O( h^{3/2-\varepsilon} ),
\end{align*}
for any $\varepsilon >0$.  We demonstrate the performance of the HDG method with numerical experiments in 2D and 3D in Section \ref{sec:numerics}.

\section{Background: The Optimality System and Regularity}
\label{sec:background}

To begin, we review some fundamental results concerning the optimality system for the control problem and the regularity of the solution in 2D polygonal domains.

Throughout the paper we adopt the standard notation $W^{m,p}(\Omega)$ for Sobolev spaces on $\Omega$ with norm $\|\cdot\|_{m,p,\Omega}$ and seminorm $|\cdot|_{m,p,\Omega}$ . We denote $W^{m,2}(\Omega)$ by $H^{m}(\Omega)$ with norm $\|\cdot\|_{m,\Omega}$ and seminorm $|\cdot|_{m,\Omega}$. Also, $H_0^1(\Omega)=\{v\in H^1(\Omega):v=0 \;\mbox{on}\; \partial \Omega\}$.  We denote the $L^2$-inner products on $L^2(\Omega)$ and $L^2(\Gamma)$ by
\begin{align*}
(v,w) &= \int_{\Omega} vw  \quad \forall v,w\in L^2(\Omega),\\
\left\langle v,w\right\rangle &= \int_{\Gamma} vw  \quad\forall v,w\in L^2(\Gamma).
\end{align*}
Define the space $H(\text{div};\Omega)$ as
\begin{align*}
H(\text{div},\Omega) = \{\bm{v}\in [L^2(\Omega)]^2, \nabla\cdot \bm{v}\in L^2(\Omega)\}.
\end{align*}

To avoid the the variational difficulty we follow the strategy introduced by Wei Gong and Ningning Yan \cite{MR2806572} and consider a mixed formulation of the optimality system.  Introduce two flux variables $\bm{q} = -\nabla y$ and $\bm{p} = -\nabla z$.  The mixed weak form of  \eqref{boundary_pro_a}-\eqref{boundary_pro_d} is
\begin{subequations}\label{mixed}
	\begin{align}
	(\bm{q},\bm{r})-(y,\nabla\cdot\bm{r})+\left\langle u, \bm{r}\cdot \bm{n}\right\rangle&=0, \label{mixed_a}\\
	(\nabla\cdot \bm{q}, w) &= (f,w),  \label{mixed_b}\\
	(\bm{p},\bm{r})-(z,\nabla\cdot\bm{r})&=0, \label{mixed_c}\\
	(\nabla\cdot \bm{p}, w) -(y,w)&= (y_d,w), \label{mixed_d}\\
	\left\langle\gamma u + \bm{p}\cdot \bm{n}, \xi \right\rangle &=0, \label{mixed_e}
	\end{align}
\end{subequations}
for all $(\bm{r},w,\xi)\in H(\text{div},\Omega)\times L^2(\Omega)\times L^2(\Gamma)$.


One of the main reasons that Dirichlet boundary control problem can be challenging numerically is that the solution can have very low regularity, and this restricts the convergence rates of finite element and DG methods.  In order to prove a superlinear convergence rate for the optimal control for the HDG method in \ref{sec:analysis}, we assume the solution has the following fractional Sobolev regularity:
\begin{equation}\label{eqn:regularity1}
u \in H^{r_u}(\Gamma),  \quad  y \in H^{r_y}(\Omega),  \quad  z \in H^{r_z}(\Omega),  \quad  \bm{q} \in H^{r_{\bm q}}(\Omega),  \quad  \bm{p} \in H^{r_{\bm p}}(\Omega),
\end{equation}
with
\begin{equation}\label{eqn:regularity2}
r_u > 1,  \quad  r_y > 1,  \quad  r_z > 2,  \quad  r_{\bm q} > 1/2,  \quad  r_{\bm p} > 1.
\end{equation}
We require $ r_{\bm q} > 1/2 $ in order to guarantee $ q $ has a well-defined trace on the boundary $ \Gamma $.  We note that it may be possible to use the techniques in \cite{MR3508837} to lower the regularity requirement on $ \bm q $.  We leave this to be considered elsewhere.

For a 2D convex polygonal domain and $ f = 0 $, we use a recent regularity result of Mateos and Neitzel \cite{MR3465458} below to give conditions on the domain and $ y_d $ to guarantee the solution has the above regularity.  For a higher dimensional convex polyhedral domain, the regularity theory is much more complicated, and we do not attempt to provide conditions to guarantee the above regularity in this work.
\begin{theorem}[\cite{MR3465458}, Lemma 3 and Corollary 1]\label{thm:regularity}
	Suppose $ f = 0 $ and $ \Omega \subset \mathbb{R}^2 $ is a bounded convex domain with polygonal boundary $ \Gamma $.  Let $ \omega \in [ \pi/3, \pi ) $ be the largest interior angle of $\Gamma$, and define $ p_{\Omega}$, $ r_\Omega $ by
	\[ p_\Omega=\frac{2}{ 2-\pi/\max\{\omega, \pi/2\} } \in (2, \infty], \]
	and
	\[ r_\Omega=1+\frac{\pi}{\omega} \in (2, 4]. \]
	If $ y_d \in L^p(\Omega) \cap H^{r-2}(\Omega) $ for all $ p < p_\Omega $ and $ r < r_\Omega $, then the solution $ (u,y,z) $ satisfies
	\begin{align*}
	  u &\in H^{r-3/2}(\Gamma) \cap W^{1-1/p,p}(\Gamma),\\
	  y &\in H^{r-1}(\Omega) \cap W^{1,p}(\Omega),\\
	  z &\in H^1_0(\Omega) \cap H^r(\Omega) \cap W^{2,p}(\Omega)
	\end{align*}
	%
	%
	for all
	\[ p < p_\Omega,  \quad  r < \min\{ 3, r_\Omega \}. \]
\end{theorem}

We also require the regularity for the flux variables $ \bm{q} = -\nabla y $ and $ \bm p= -\nabla z $.
\begin{corollary}\label{cor:regularity_flux}
	Under the assumptions of Theorem \ref{thm:regularity}, the flux variables $ \bm{q} = -\nabla y $ and $ \bm p = -\nabla z $ satisfy
	\[ \bm{q} \in H^{r-2}(\Omega) \cap H(\mathrm{div},\Omega),  \quad  \bm{p} \in H^{r-1}(\Omega) \cap H(\mathrm{div},\Omega)  \]
	%
	for all $ r < \min\{ 3, r_\Omega \} $.
\end{corollary}
\begin{proof}
	We treat the optimal control $ u $ as known, and then $ (y,\bm{q}) $ satisfy the weak mixed formulation \eqref{mixed_a}-\eqref{mixed_b}.  Since $ u \in H^{1/2}(\Gamma) $, the standard theory for this mixed problem gives $ \bm{q} \in H(\text{div},\Omega) $.  Taking $ \bm{r} $ smooth and integrating by parts in \eqref{mixed_a} gives $ \bm{q} = -\nabla y $, and therefore the fractional Sobolev regularity for $ \bm{q} $ follows directly from Theorem \ref{thm:regularity}.  The regularity for $ \bm{p} $ follows similarly.
\end{proof}

The regularity for the flux variable $ \bm{q} = -\nabla y $ is low; Theorem \ref{cor:regularity_flux} only guarantees $ \bm{q} \in H^{r_{\bm q}} $ for some $ 0 < r_{\bm q} < 1 $.  For the HDG approximation theory, we need the regularity condition $ r_{\bm q} > 1/2 $.  We can guarantee this condition by restricting the maximum interior angle $ \omega $.  Specifically, if if $ y_d $ has the required smoothness and $ \omega $ satisfies
\[ \omega \in [\pi/3, 2 \pi/3), \]
then $ r_\Omega \in (5/2,4] $ and we are guaranteed $ \bm{q} \in H^{r_{\bm q}} $ for some $ r_{\bm q} > 1/2 $.

Also, when we restrict $ \omega \in [\pi/3, 2 \pi/3) $ as above, this guarantees $ u \in H^{r_u} $ for some $ 1 < r_u < 3/2 $ and furthermore the regularity assumption \eqref{eqn:regularity1}-\eqref{eqn:regularity2} is satisfied.  For a rectangular domain, we have $ p_\Omega = \infty $ and $ r_\Omega = 3 $.  Therefore if $ y_d \in H^1(\Omega) \cap L^\infty(\Omega) $ we are guaranteed the fractional Sobolev regularity
\[ r_u = \frac{3}{2} - \varepsilon,  \quad  r_y = 2 - \varepsilon,  \quad  r_z = 3 - \varepsilon,  \quad  r_{\bm q} = 1 - \varepsilon,  \quad  r_{\bm p} = 2 - \varepsilon \]
for any $ \varepsilon > 0 $.

\section{HDG Formulation and Implementation}
\label{sec:HDG}

A mixed method can avoid the variational difficulty by the introducing flux variables $\bm{q}$ and $\bm{p}$ and the equation for the optimal control \eqref{mixed_e}. However, these two additional vector variables will increase the computational cost, even if the lowest order RT method is used.

We introduce an HDG method for the optimality system \eqref{boundary_pro} to take advantage of the mixed formulation and also reduce the computational cost compared to standard mixed methods.  Specifically, we introduce the flux variables but eliminate them before we solve the global equation; this significantly reduces the degrees of freedom.

HDG methods were proposed by Cockburn et al.\ in \cite{MR2485455} as an improvement of tradition discontinuous Galerkin methods and have many applications; see, e.g., 
\cite{MR2772094,MR2513831,MR2558780,MR3626531,MR3522968,MR3477794,MR3463051,MR3452794,MR3343926,MR2796169}.  The approximate scalar variable and flux are expressed in an element-by-element fashion in terms of an approximate trace of the scalar variable along the element boundary. Then, a unique value for the trace at inter-element boundaries is obtained by enforcing flux continuity. This leads to a global equation system in terms of the approximate boundary traces only. The high number of globally coupled degrees of freedom is  significantly reduced compared to  other DG methods and standard mixed methods.



Before we introduce the HDG method, we first set some notation.  Let $ \{ \mathcal T_h \} $ be a conforming quasi-uniform polyhedral mesh of $ \Omega $. We denote by $\partial \mathcal{T}_h$ the set $\{\partial K: K\in \mathcal{T}_h\}$. For an element $K$ of the collection  $\mathcal{T}_h$, let $e = \partial K \cap \Gamma$ denote the boundary face of $ K $ if the $d-1$ Lebesgue measure of $e$ is non-zero. For two elements $K^+$ and $K^-$ of the collection $\mathcal{T}_h$, let $e = \partial K^+ \cap \partial K^-$ denote the interior face between $K^+$ and $K^-$ if the $d-1$ Lebesgue measure of $e$ is non-zero. Let $\varepsilon_h^o$ and $\varepsilon_h^{\partial}$ denote the set of interior and boundary faces, respectively. We denote by $\varepsilon_h$ the union of  $\varepsilon_h^o$ and $\varepsilon_h^{\partial}$. We finally introduce
\begin{align*}
(w,v)_{\mathcal{T}_h} = \sum_{K\in\mathcal{T}_h} (w,v)_K,   \quad\quad\quad\quad\left\langle \zeta,\rho\right\rangle_{\partial\mathcal{T}_h} = \sum_{K\in\mathcal{T}_h} \left\langle \zeta,\rho\right\rangle_{\partial K}.
\end{align*}

Let $\mathcal{P}^k(D)$ denote the set of polynomials of degree at most $k$ on a domain $D$.  We introduce the discontinuous finite element spaces
\begin{align}
\bm{V}_h  &:= \{\bm{v}\in [L^2(\Omega)]^d: \bm{v}|_{K}\in [\mathcal{P}^k(K)]^d, \forall K\in \mathcal{T}_h\},\\
{W}_h  &:= \{{w}\in L^2(\Omega): {w}|_{K}\in \mathcal{P}^{k+1}(K), \forall K\in \mathcal{T}_h\},\\
{M}_h  &:= \{{\mu}\in L^2(\mathcal{\varepsilon}_h): {\mu}|_{e}\in \mathcal{P}^k(e), \forall e\in \varepsilon_h\}.
\end{align}
The space $ W_h $ is for scalar variables, while $ \bm{V}_h $ is for flux variables and $ M_h $ is for boundary trace variables.  Note that the polynomial degree for the scalar variables is one order higher than the polynomial degree for the other variables.  Also, the boundary trace variables will be used to eliminate the state and flux variables from the coupled global equations, thus substantially reducing the number of degrees of freedom.

Let  $M_h(o)$ and $M_h(\partial)$ denote the spaces defined in the same way as $M_h$, but with $\varepsilon_h$ replaced by $\varepsilon_h^o$ and $\varepsilon_h^{\partial}$, respectively. Note that $M_h$ consists of functions which are continuous inside the faces (or edges) $e\in \varepsilon_h$ and discontinuous at their borders. In addition, for any function $w\in W_h$ we use $\nabla w$ to denote the piecewise gradient on each element $K\in \mathcal T_h$. A similar convention applies to the divergence operator $\nabla\cdot\bm r$ for all $\bm r\in \bm V_h$.

\subsection{The HDG Formulation}
To approximate the solution of the mixed weak form \eqref{boundary_pro_a}-\eqref{boundary_pro_d} of the optimality system, the HDG method seeks approximate fluxes ${\bm{q}}_h,{\bm{p}}_h \in \bm{V}_h $, states $ y_h, z_h \in W_h $, interior element boundary traces $ \widehat{y}_h^o,\widehat{z}_h^o \in M_h(o) $, and boundary control $ u_h \in M_h(\partial)$ satisfying
%
%
\begin{subequations}\label{HDG_discrete2}%
	\begin{align}
	(\bm{q}_h, \bm{r_1})_{{\mathcal{T}_h}}- (y_h, \nabla\cdot \bm{r_1})_{{\mathcal{T}_h}}+\langle \widehat y_h^o, \bm{r_1}\cdot \bm{n} \rangle_{\partial{{\mathcal{T}_h}}\backslash {\varepsilon_h^{\partial}}}+ \langle u_h, \bm{r_1}\cdot \bm{n} \rangle_{{\varepsilon_h^{\partial}}} &=0, \label{HDG_discrete2_a}\\
	-(\bm{q}_h, \nabla w_1)_{{\mathcal{T}_h}}
	+\langle\widehat{\bm{q}}_h\cdot\bm{n}, w_1 \rangle_{\partial{{\mathcal{T}_h}}}&=(f, w_1)_{{\mathcal{T}_h}}, \label{HDG_discrete2_b}
	\end{align}
	for all $(\bm{r_1},w_1)\in \bm{V}_h\times W_h$,
	\begin{align}
	(\bm{p}_h, \bm{r_2})_{{\mathcal{T}_h}}- (z_h, \nabla\cdot \bm{r_2})_{{\mathcal{T}_h}}+\langle \widehat{z}_h^o, \bm{r_2}\cdot \bm{n} \rangle_{\partial{{\mathcal{T}_h}}\backslash {\varepsilon_h^{\partial}}} &=0, \label{HDG_discrete2_c}\\
	-(\bm{p}_h, \nabla w_2)_{{\mathcal{T}_h}}
	+\langle\widehat{\bm{p}}_h\cdot \bm{n}, w_2 \rangle_{\partial{{\mathcal{T}_h}}}-({y}_h,  w_2)_{{\mathcal{T}_h}}&=-(y_d, w_2)_{{\mathcal{T}_h}}, \label{HDG_discrete2_d}
	\end{align}
	for all $(\bm{r_2},w_2)\in \bm{V}_h\times W_h$,
	\begin{align}
	\langle\widehat{\bm{q}}_h\cdot \bm{n}, \mu_1 \rangle_{\partial\mathcal{T}_{h}\backslash {{\varepsilon_h^{\partial}}}}&=0, \label{HDG_discrete2_e}
	\end{align}
	for all $\mu_1\in M_h(o)$,
	\begin{align}
	\langle\widehat{\bm{p}}_h\cdot \bm{n}, \mu_2 \rangle_{ \partial\mathcal{T}_{h}\backslash {{\varepsilon_h^{\partial}}}}&=0, \label{HDG_discrete2_g}
	\end{align}
	for all $\mu_2\in M_h(o)$, and
	\begin{align}
	\langle u_h, \mu_3 \rangle_{{\varepsilon_h^{\partial}}}+\langle \gamma^{-1}\widehat{\bm{p}}_h\cdot \bm{n}, \mu_3\rangle_{{{\varepsilon_h^{\partial}}}} &=0, \label{HDG_discrete2_f}
	\end{align}
	for all $\mu_3\in M_h(\partial)$.
	
	The numerical traces on $\partial\mathcal{T}_h$ are defined as
	\begin{align}
	\widehat{\bm{q}}_h\cdot \bm n &=\bm q_h\cdot\bm n+h^{-1}(P_My_h-\widehat y_h^o)   \quad \mbox{on} \; \partial \mathcal{T}_h\backslash\varepsilon_h^\partial, \label{HDG_discrete2_h}\\
	\widehat{\bm{q}}_h\cdot \bm n &=\bm q_h\cdot\bm n+h^{-1} (P_My_h-u_h)  \quad \mbox{on}\;  \varepsilon_h^\partial, \label{HDG_discrete2_i}\\
	\widehat{\bm{p}}_h\cdot \bm n &=\bm p_h\cdot\bm n+h^{-1}(P_Mz_h-\widehat z_h^o)\quad \mbox{on} \; \partial \mathcal{T}_h\backslash\varepsilon_h^\partial,\label{HDG_discrete2_j}\\
	\widehat{\bm{p}}_h\cdot \bm n &=\bm p_h\cdot\bm n+h^{-1} P_Mz_h\quad\quad\quad\quad\mbox{on}\;  \varepsilon_h^\partial,\label{HDG_discrete2_k}
	\end{align}
\end{subequations}
where $P_M$ denotes the standard $L^2$-orthogonal projection from $L^2(\varepsilon_h)$ onto $M_h$.  This completes the formulation of the HDG method.

The HDG formulation with $ h^{-1} $ stabilization, polynomial degree $ k+1 $ for the scalar unknown, and polynomial degree $ k $ for the other unknowns was originally introduced by Lehrenfeld in \cite{Lehrenfeld10}.

\subsection{Implementation}
To arrive at the HDG formulation we implement numerically, we insert  \eqref{HDG_discrete2_h}-\eqref{HDG_discrete2_k} into \eqref{HDG_discrete2_a}-\eqref{HDG_discrete2_f}, and find after some simple manipulations that
\[ ({\bm{q}}_h,{\bm{p}}_h, y_h,z_h,{\widehat{y}}_h^o,{\widehat{z}}_h^o,u_h)\in \bm{V}_h\times\bm{V}_h\times W_h \times W_h\times M_h(o)\times M_h(o)\times M_h(\partial) \]
is the solution of the following weak formulation:
\begin{subequations}\label{imple}
	\begin{align}
	(\bm{q}_h, \bm{r_1})_{{\mathcal{T}_h}}- (y_h, \nabla\cdot \bm{r_1})_{{\mathcal{T}_h}}+\langle \widehat{y}_h^o, \bm{r_1}\cdot \bm{n} \rangle_{\partial{{\mathcal{T}_h}}\backslash \varepsilon_h^{\partial}}+ \langle u_h, \bm{r_1}\cdot \bm{n} \rangle_{\varepsilon_h^{\partial}} &=0, \label{imple_a}\\
	(\bm{p}_h, \bm{r_2})_{{\mathcal{T}_h}}- (z_h, \nabla\cdot \bm{r_2})_{{\mathcal{T}_h}}+\langle \widehat{z}_h^o, \bm{r_2}\cdot \bm{n} \rangle_{\partial{{\mathcal{T}_h}}\backslash \varepsilon_h^{\partial}} &=0, \label{imple_b}\\
	(\nabla\cdot\bm{q}_h,  w_1)_{{\mathcal{T}_h}}
	+\langle h^{-1}P_M y_h, w_1 \rangle_{\partial{{\mathcal{T}_h}}}-\langle h^{-1}\widehat{y}_h^o, w_1 \rangle_{\partial{{\mathcal{T}_h}}\backslash \varepsilon_h^{\partial}} \quad & \\
	-\langle h^{-1} u_h, w_1 \rangle_{\varepsilon_h^\partial}&=(f, w_1)_{{\mathcal{T}_h}}, \label{imple_c}\\
	(\nabla\cdot\bm{p}_h,  w_2)_{{\mathcal{T}_h}}
	+\langle h^{-1} P_Mz_h, w_2 \rangle_{\partial{{\mathcal{T}_h}}}-\langle h^{-1} \widehat{z}_h^o, w_2 \rangle_{\partial{{\mathcal{T}_h}}\backslash \varepsilon_h^{\partial}} \quad & \\
	-(y_h,  w_2)_{{\mathcal{T}_h}}&=-(y_d, w_2)_{{\mathcal{T}_h}}, \label{imple_d}\\
	\langle{\bm{q}_h}\cdot \bm{n}, \mu_1 \rangle_{\partial\mathcal{T}_{h}\backslash {\varepsilon_h^{\partial}}}+\langle h^{-1} y_h ,\mu_1 \rangle_{\partial\mathcal{T}_{h}\backslash {\varepsilon_h^{\partial}}}-\langle h^{-1} \widehat{y}_h^{o} ,\mu_1 \rangle_{\partial\mathcal{T}_{h}\backslash {\varepsilon_h^{\partial}}}&=0, \label{imple_e}\\
	\langle {\bm{p}_h}\cdot \bm{n}, \mu_2\rangle_{ \partial\mathcal{T}_{h}\backslash {\varepsilon_h^{\partial}}}+\langle h^{-1}  z_h, \mu_2\rangle_{ \partial\mathcal{T}_{h}\backslash {\varepsilon_h^{\partial}}}-\langle h^{-1} \widehat{z}_h^o, \mu_2\rangle_{ \partial\mathcal{T}_{h}\backslash {\varepsilon_h^{\partial}}} &=0,
	\label{imple_f}\\
	\langle u_h, \mu_3 \rangle_{\varepsilon_h^{\partial}}+\langle \gamma^{-1}{\bm{p}_h}\cdot \bm{n}, \mu_3\rangle_{{\varepsilon_h^{\partial}}}+\langle \gamma^{-1}h^{-1} z_h, \mu_3\rangle_{{\varepsilon_h^{\partial}}}&=0, \label{imple_g}
	\end{align}
\end{subequations}
for all  $({\bm{r}_1},{\bm{r}_2},w_1,w_2,\mu_1,\mu_2,\mu_3)\in \bm{V}_h\times\bm{V}_h\times W_h \times W_h\times M_h(o)\times M_h(o)\times M_h(\partial)$.

\subsubsection{Matrix equations}

Assume $\bm{V}_h = \mbox{span}\{\bm\varphi_i\}_{i=1}^{N_1}$, $W_h=\mbox{span}\{\phi_i\}_{i=1}^{N_2}$, $M_h^{o}=\mbox{span}\{\psi_i\}_{i=1}^{N_3} $, and $M_h^{\partial}=\mbox{span}\{\psi_i\}_{i=1+N_3}^{N_4}$. Then
\begin{equation}\label{expre}
\begin{split}
&\bm q_{h}= \sum_{j=1}^{N_1}q_{j}\bm\varphi_j,  \quad
\bm p_{h} =  \sum_{j=1}^{N_1}p_{j}\bm\varphi_j,  \quad
y_h = \sum_{j=1}^{N_2}y_{j}\phi_j, \quad z_h =  \sum_{j=1}^{N_2} z_{j}\phi_j,\\
& \widehat{y}_h^o = \sum_{j=1}^{N_3}\alpha_{j}\psi_{j},  \quad  \widehat{z}_h^o = \sum_{j=1}^{N_3}\gamma_{j}\psi_{j}, \quad u_h = \sum_{j=1+N_3}^{N_4} \beta_{j} {\psi}_j.
\end{split}
\end{equation}
Substitute \eqref{expre} into \eqref{imple_a}-\eqref{imple_g} and use the corresponding  test functions to test \eqref{imple_a}-\eqref{imple_g}, respectively, to obtain the matrix equation
\begin{align}\label{system_equation}
\begin{bmatrix}
A_1  &0 &-A_2&0  & A_8&0&A_{9} \\
0 & A_1 &0 &-A_2& 0  &A_8& 0 \\
A_2^T &0&A_5&0&-A_{10}  &0& -A_{11}\\
0 &A_2^T & -A_4 &A_5 &0&-A_{10} &0\\
A_8^T & 0 &A_{10}^T&0 &A_{11}&0&0  \\
0& A_8^T &0&A_{10}^T &0&A_{11} &0 \\
0 & \gamma^{-1}A_{12} &0&\gamma^{-1}A_{13}&0&0 & A_{14}
\end{bmatrix}
\left[ {\begin{array}{*{20}{c}}
	\mathfrak{q}\\
	\mathfrak{p}\\
	\mathfrak{y}\\
	\mathfrak{z}\\
	\mathfrak{\widehat y}\\
	\mathfrak{\widehat z}\\
	\mathfrak{u}
	\end{array}} \right]
=\left[ {\begin{array}{*{20}{c}}
	0\\
	0\\
	b_1\\
	-b_2\\
	0\\
	0\\
	0\\
	\end{array}} \right].
\end{align}
Here, $\mathfrak{q},\mathfrak{p},\mathfrak{y},\mathfrak{z},\mathfrak{\widehat y},\mathfrak{\widehat z},\mathfrak{u}$ are the coefficient vectors for $\bm q_h,\bm p_h,y_h,z_h,\widehat y_h^o, \widehat z_h^o, u_h$, respectively, and
%
%
\begin{gather*}
A_1= [(\bm\varphi_j,\bm\varphi_i )_{\mathcal{T}_h}],  \:\:  A_2 = [(\phi_j,\nabla\cdot\bm{\varphi_i})_{\mathcal{T}_h}],  \:\:  A_3 = [(\psi_j,\bm{\varphi}_i\cdot \bm n)_{\mathcal{T}_h}], \:\:  A_4 = [(\phi_j,\phi_i)_{\mathcal{T}_h}],\\
A_5 = [\langle  h^{-1}P_M\phi_j, \phi_i \rangle_{\partial{{\mathcal{T}_h}}}], \quad  A_6 =  [\langle h^{-1}\psi_j,\psi_i\rangle_{\partial\mathcal{T}_h}], \quad  A_7=
[\langle h^{-1}\psi_j,{\varphi_i}\rangle_{\partial\mathcal{T}_h}], \\
b_1 = [(f,\phi_i )_{\mathcal{T}_h}], \quad  b_2 = [(y_d,\phi_i )_{\mathcal{T}_h}].
\end{gather*}
%
The remaining matrices $A_{8}-A_{14}$ are constructed by extracting the corresponding rows and columns from $A_3 $, $A_6$, and $A_7$. In the actual computation, to save memory we do not assemble the large matrix in equation \eqref{system_equation}.
%

Equation \eqref{system_equation} can be rewritten as
\begin{align}\label{system_equation2}
\begin{bmatrix}
B_1 & B_2&B_3\\
-B_2^T & B_4&B_5\\
B_6&B_7&B_8\\
\end{bmatrix}
\left[ {\begin{array}{*{20}{c}}
	\bm{\alpha}\\
	\bm{\beta}\\
	\bm{\gamma}
	\end{array}} \right]
=\left[ {\begin{array}{*{20}{c}}
	0\\
	b\\
	0
	\end{array}} \right],
\end{align}
where $\bm{\alpha}=[\mathfrak{q};\mathfrak{p}]$, $\bm{\beta}=[\mathfrak{y};\mathfrak{z}]$, $\bm{\gamma}=[\mathfrak{\widehat y};\mathfrak{\widehat z};\mathfrak{u}]$, $ b = [ b_1; -b_2 ] $, and $\{B_i\}_{i=1}^8$ are the corresponding blocks of the coefficient matrix in \eqref{system_equation}.

Due to the discontinuous nature of the approximation spaces $\bm{V_h}$ and ${W_h}$, the first two equations of  \eqref{system_equation2} can be used to eliminate both $\bm{\alpha}$ and $\bm{\beta}$ in an element-by-element fashion.  As a consequence, we can write system \eqref{system_equation2}  as
\begin{align}\label{local_solver}
\left[ {\begin{array}{*{20}{c}}
	\bm{\alpha}\\
	\bm{\beta}
	\end{array}} \right]= \begin{bmatrix}
G_1 & H_1\\
G_2 & H_2
\end{bmatrix}
\left[ {\begin{array}{*{20}{c}}
	\bm \gamma\\
	b
	\end{array}} \right]
\end{align}
and
\begin{align}
B_6\bm{\alpha}+B_7\bm{\beta}+B_8\bm{\gamma} = 0.
\end{align}
We provide details on the element-by-element construction of $G_1,G_2$ and $H_1,H_2$ in the appendix.  Next, we eliminate both $\bm{\alpha}$ and $\bm{\beta}$ to obtain a reduced globally coupled equation for $\bm{\gamma}$ only:
\begin{align}\label{global_eq}
\mathbb{K}\bm{\gamma} = \mathbb{F},
\end{align}
where
\begin{align*}
\mathbb{K}= B_6G_1+B_7G_2+B_8\quad\text{and}\quad\mathbb{F} = B_6H_1+B_7H_2.
\end{align*}
Once $\bm{\gamma}$ is available, both $\bm{\alpha}$ and $\bm{\beta}$ can be recovered from \eqref{local_solver}.
%
%

\begin{remark}
	For HDG methods, the standard approach is to first compute the local solver independently on each element and then assemble the global system.  The process we follow here is to first assemble the global system and then reduce its dimension by simple block-diagonal algebraic operations.  The two approaches are equivalent.
\end{remark}

Equation \eqref{local_solver} says we can express the approximate the scalar state variable and corresponding fluxes in terms of the approximate traces on the element boundaries.  The global equation \eqref{global_eq} only involves the approximate traces. Therefore, the high number of globally coupled degrees of freedom in the HDG method is significantly reduced. This is one excellent feature of HDG methods.

\section{Error Analysis}
\label{sec:analysis}

Next, we provide a convergence analysis of the above HDG method for the Dirichlet boundary control problem.  Throughout this section, we assume $ \Omega $ is a bounded convex polyhedral domain and we also assume the regularity condition \eqref{eqn:regularity1}-\eqref{eqn:regularity2} is satisfied.  For the 2D case, recall Section \ref{sec:background} provides conditions on $ \Omega $ and $ y_d $ guaranteeing the required regularity.

\subsection{Main result}

First, we present the following main theoretical result of this work.  Recall we assume the fractional Sobolev regularity exponents satisfy
\[ r_u > 1,  \quad  r_y > 1,  \quad  r_z > 2,  \quad  r_{\bm q} > 1/2,  \quad  r_{\bm p} > 1. \]
\begin{theorem}\label{main_res}
	For
	\begin{equation*}\label{eqn:s_rates}
	s_{y} = \min\{r_{y}, k+2\},  \:\:  s_{z} = \min\{r_{z}, k+2\},  \:\:  s_{\bm q} = \min\{r_{\bm q}, k+1\},  \:\:  s_{\bm p} = \min\{r_{\bm p}, k+1\},
	\end{equation*}
	%
	we have
	\begin{align*}
	\norm{u-u_h}_{\varepsilon_h^\partial}&\lesssim h^{s_{\bm p}-\frac 1 2}\norm{\bm p}_{s_{\bm p},\Omega} +  h^{s_{z}-\frac 3 2}\norm{z}_{s_{z},\Omega} + h^{s_{\bm q}+\frac 1 2}\norm{\bm q}_{s_{\bm q},\Omega} + h^{s_{y}-\frac 12 }\norm{y}_{s_{y},\Omega},\\
	\norm {y-y_h}_{\mathcal T_h} &\lesssim h^{s_{\bm p}-\frac 1 2}\norm{\bm p}_{s_{\bm p},\Omega} +  h^{s_{z}-\frac 3 2}\norm{z}_{s_{z},\Omega} + h^{s_{\bm q}+\frac 1 2}\norm{\bm q}_{s_{\bm q},\Omega} + h^{s_{y}-\frac 12 }\norm{y}_{s_{y},\Omega}\\
	\norm {\bm q - \bm q_h}_{\mathcal T_h} &\lesssim h^{s_{\bm p}-1}\norm{\bm p}_{s_{\bm p},\Omega} +  h^{s_{z}-2}\norm{z}_{s_{z},\Omega} + h^{s_{\bm q}}\norm{\bm q}_{s_{\bm q},\Omega} + h^{s_{y}-1 }\norm{y}_{s_{y},\Omega},\\
	\norm {\bm p - \bm p_h}_{\mathcal T_h}  &\lesssim h^{s_{\bm p}-\frac 1 2}\norm{\bm p}_{s_{\bm p},\Omega} +  h^{s_{z}-\frac 3 2}\norm{z}_{s_{z},\Omega} + h^{s_{\bm q}+\frac 1 2}\norm{\bm q}_{s_{\bm q},\Omega} + h^{s_{y}-\frac 12 }\norm{y}_{s_{y},\Omega},\\
	\norm {z - z_h}_{\mathcal T_h} & \lesssim  h^{s_{\bm p}-\frac 1 2}\norm{\bm p}_{s_{\bm p},\Omega} +  h^{s_{z}-\frac 3 2}\norm{z}_{s_{z},\Omega} + h^{s_{\bm q}+\frac 1 2}\norm{\bm q}_{s_{\bm q},\Omega} + h^{s_{y}-\frac 12 }\norm{y}_{s_{y},\Omega}.
	\end{align*}
\end{theorem}

Using the regularity results for the 2D case presented in Section \ref{sec:background}, we obtain the following result.
\begin{corollary}\label{main_reslut}
	Suppose $ d = 2 $, $ f = 0 $, and $ k = 1 $.  Let $ \omega \in [ \pi/3, 2\pi/3 ) $ be the largest interior angle of $\Gamma$, and define $ p_\Omega$, $ r_\Omega $ by
	\[ p_\Omega=\frac{2}{ 2-\pi/\max\{\omega, \pi/2\} } \in (4, \infty],  \quad  r_\Omega=1+\frac{\pi}{\omega} \in (5/2, 4]. \]
	If $ y_d \in L^p(\Omega) \cap H^{r-2}(\Omega) $ for all $ p < p_\Omega $ and $ r < r_\Omega $, then for any $ r < \min\{ 3, r_\Omega \} $ we have
	%
	%
	\begin{align*}
	\norm{u-u_h}_{\varepsilon_h^\partial}&\lesssim h^{r - \frac 3 2} (\norm{\bm p}_{H^{r-1}(\Omega)} +  \norm{z}_{H^{r}(\Omega)} + \norm{\bm q}_{H^{r-2}(\Omega)} + \norm{y}_{H^{r-1}(\Omega)}),\\
	\norm{y-y_h}_{\mathcal T_h}&\lesssim h^{r - \frac 3 2} (\norm{\bm p}_{H^{r-1}(\Omega)} +  \norm{z}_{H^{r}(\Omega)} + \norm{\bm q}_{H^{r-2}(\Omega)} + \norm{y}_{H^{r-1}(\Omega)}),\\
	\norm {\bm q - \bm q_h}_{\mathcal T_h}  &\lesssim h^{r-2} (\norm{\bm p}_{H^{r-1}(\Omega)} +  \norm{z}_{H^{r}(\Omega)} + \norm{\bm q}_{H^{r-2}(\Omega)} + \norm{y}_{H^{r-1}(\Omega)}),\\
	\norm {\bm p - \bm p_h}_{\mathcal T_h}   &\lesssim h^{r - \frac 3 2} (\norm{\bm p}_{H^{r-1}(\Omega)} +  \norm{z}_{H^{r}(\Omega)} + \norm{\bm q}_{H^{r-2}(\Omega)} + \norm{y}_{H^{r-1}(\Omega)}),\\
	\norm {z - z_h}_{\mathcal T_h}   & \lesssim  h^{r - \frac 3 2} (\norm{\bm p}_{H^{r-1}(\Omega)} +  \norm{z}_{H^{r}(\Omega)} + \norm{\bm q}_{H^{r-2}(\Omega)} + \norm{y}_{H^{r-1}(\Omega)}).
	\end{align*}
	%
\end{corollary}
Note that $ \min\{ 3, r_\Omega \} $ is always greater than 5/2, which guarantees a superlinear convergence rate for all variables except $ \bm{q} $.  Also, if $ \Omega $ is a rectangle (i.e., $ \omega = \pi/2 $) and $y_d\in H^1(\Omega) \cap L^\infty(\Omega)$, then $ r_\Omega = 3 $ and we obtain an $ O(h^{3/2 - \varepsilon}) $ convergence rate for $ u $, $ y $, $ z $, and $ \bm{p} $, and an $ O(h^{1 - \varepsilon}) $ convergence rate for $ \bm{q} $ for any $ \varepsilon > 0 $.

\subsection{Preliminary material}
\label{sec:Projectionoperator}

Before we prove the main result, we discuss $ L^2 $ projections, an HDG operator $ \mathscr B $, and the well-posedness of the HDG equations.

We first define the standard $L^2$ projections $\bm\Pi :  [L^2(\Omega)]^d \to \bm V_h$, $\Pi :  L^2(\Omega) \to W_h$, and $P_M:  L^2(\varepsilon_h) \to M_h$, which satisfy
\begin{equation}\label{L2_projection}
\begin{split}
(\bm\Pi \bm q,\bm r)_{K} &= (\bm q,\bm r)_{K} ,\qquad \forall \bm r\in [{\mathcal P}_{k}(K)]^d,\\
(\Pi u,w)_{K}  &= (u,w)_{K} ,\qquad \forall w\in \mathcal P_{k+1}(K),\\
\left\langle P_M m, \mu\right\rangle_{ e} &= \left\langle  m, \mu\right\rangle_{e }, \quad\;\;\; \forall \mu\in \mathcal P_{k}(e).
\end{split}
\end{equation}
In the analysis, we use the following classical results:
\begin{subequations}\label{classical_ine}
	\begin{align}
	\norm {\bm q -\bm{\Pi q}}_{\mathcal T_h} &\le Ch^{s_{\bm q}} \norm{\bm q}_{s_{\bm q},\Omega},\qquad\norm {y -{\Pi y}}_{\mathcal T_h} \le Ch^{s_{y}} \norm{y}_{s_{y},\Omega},\\
	\norm {y -{\Pi y}}_{\partial\mathcal T_h} &\le Ch^{s_{y}-\frac 1 2} \norm{y}_{s_{y},\Omega},
	\quad\norm {\bm q\cdot \bm n -\bm{\Pi q}\cdot \bm n}_{\partial \mathcal T_h} \le Ch^{s_{\bm q}-\frac 12} \norm{\bm q}_{s_{\bm q},\Omega},\\
	\norm {w}_{\partial \mathcal T_h} &\le Ch^{-\frac 12} \norm {w}_{ \mathcal T_h}, \quad\quad \forall w\in W_h,
	\end{align}
\end{subequations}
where $ s_{\bm q} $ and $ s_y $ are defined in Theorem \ref{main_res}.  We have the same projection error bounds for $\bm p$ and $z$.

To shorten lengthy equations, we define the HDG operator $ \mathscr B $ as follows:
\begin{align}
\hspace{2em}&\hspace{-2em}\mathscr  B( \bm q_h,y_h,\widehat y_h^o;\bm r_1,w_1,\mu_1)\\
&=(\bm{q}_h, \bm{r_1})_{{\mathcal{T}_h}}- (y_h, \nabla\cdot \bm{r_1})_{{\mathcal{T}_h}}+\langle \widehat{y}_h^o, \bm{r_1}\cdot \bm{n} \rangle_{\partial{{\mathcal{T}_h}}\backslash {\varepsilon_h^{\partial}}}\nonumber\\
& \quad - (\bm{q}_h, \nabla w_1)_{{\mathcal{T}_h}}+\langle\bm q_h\cdot\bm n +h^{-1} P_My_h , w_1 \rangle_{\partial{{\mathcal{T}_h}}} \nonumber\\
& \quad - \langle h^{-1} \widehat y_h^o, w_1 \rangle_{\partial{{\mathcal{T}_h}}\backslash \varepsilon_h^{\partial}} - \langle\bm q_h\cdot\bm n + h^{-1} (P_My_h-\widehat y_h^o), \mu_1 \rangle_{\partial{{\mathcal{T}_h}}\backslash\varepsilon_h^{\partial}}.\label{def_B}
\end{align}
By the definition of $ \mathscr  B $, we can rewrite the HDG formulation of the optimality system \eqref{HDG_discrete2} as follows: find $({\bm{q}}_h,{\bm{p}}_h,y_h,z_h,\widehat y_h^o,\widehat z_h^o,u_h)\in \bm{V}_h\times\bm{V}_h\times W_h \times W_h\times M_h(o)\times M_h(o)\times M_h(\partial)$  such that
\begin{subequations}\label{HDG_full_discrete}
	\begin{align}
	\mathscr B (\bm q_h,y_h,\widehat y_h^o;\bm r_1,w_1,\mu_1) &= -\langle u_h, \bm{r_1}\cdot \bm{n} - h^{-1} w_1 \rangle_{{\varepsilon_h^{\partial}}}+(f, w_1)_{{\mathcal{T}_h}}, \label{HDG_full_discrete_a}\\
	\mathscr B(\bm p_h,z_h,\widehat z_h^o;\bm r_2,w_2,\mu_2) &= (y_h -y_d,w_2)_{\mathcal T_h},\label{HDG_full_discrete_b}\\
	\gamma^{-1}\langle {\bm{p}}_h\cdot \bm{n} + h^{-1}P_M z_h, \mu_3\rangle_{{{\varepsilon_h^{\partial}}}} &= -\langle u_h, \mu_3 \rangle_{{\varepsilon_h^{\partial}}}, \label{HDG_full_discrete_e}
	\end{align}
\end{subequations}
for all $\left(\bm{r}_1, \bm{r}_2, w_1, w_2, \mu_1, \mu_2, {\mu}_3\right)\in \bm V_h\times\bm V_h\times W_h\times W_h\times M_h(o)\times M_h(o)\times M_h(\partial) $.

Next, we present a basic property of the operator $ \mathscr B $ and show the HDG equations \eqref{HDG_full_discrete} have a unique solution.
\begin{lemma}\label{energy_norm}
	For any $ ( \bm v_h, w_h, \mu_h ) \in \bm V_h \times W_h \times M_h $, we have
	\begin{align*}
	\mathscr B (\bm v_h, w_h, \mu_h;\bm v_h, w_h, \mu_h) &= (\bm{v}_h, \bm{v}_h)_{{\mathcal{T}_h}}+\langle h^{-1} (P_Mw_h- \mu_h) , P_M w_h- \mu_h\rangle_{\partial{{\mathcal{T}_h}}\backslash\varepsilon_h^{\partial}}\\
	& \quad + \langle h^{-1} P_Mw_h, P_Mw_h \rangle_{\varepsilon_h^{\partial}}.
	\end{align*}
\end{lemma}
\begin{proof}
	By the definition of $ \mathscr B $ in \eqref{def_B}, we have
	\begin{align*}
	\hspace{2em}&\hspace{-2em}  \mathscr B (\bm v_h, w_h, \mu_h;\bm v_h, w_h, \mu_h) \\
	&=(\bm{v}_h, \bm{v}_h)_{{\mathcal{T}_h}}- (w_h, \nabla\cdot \bm{v}_h)_{{\mathcal{T}_h}}+\langle \mu_h, \bm{v}_h\cdot \bm{n} \rangle_{\partial{{\mathcal{T}_h}}\backslash {\varepsilon_h^{\partial}}}- (\bm{v}_h, \nabla w_h)_{{\mathcal{T}_h}}\\
	& \quad +\langle\bm v_h\cdot\bm n +h^{-1} P_Mw_h , w_h \rangle_{\partial{{\mathcal{T}_h}}} - \langle h^{-1}   \mu_h, w_h \rangle_{\partial{{\mathcal{T}_h}}\backslash \varepsilon_h^{\partial}}\\
	& \quad - \langle\bm v_h\cdot\bm n + h^{-1} (P_Mw_h-\mu_h), \mu_h \rangle_{\partial{{\mathcal{T}_h}}\backslash\varepsilon_h^{\partial}}\\
	&=(\bm{v}_h, \bm{v}_h)_{{\mathcal{T}_h}}+\langle h^{-1} P_Mw_h , w_h \rangle_{\partial{{\mathcal{T}_h}}} - \langle h^{-1}\mu_h, w_h \rangle_{\partial{{\mathcal{T}_h}}\backslash \varepsilon_h^{\partial}}\\
	& \quad  - \langle  h^{-1} (P_Mw_h-\mu_h), \mu_h \rangle_{\partial{{\mathcal{T}_h}}\backslash\varepsilon_h^{\partial}}\\
	&=(\bm{v}_h, \bm{v}_h)_{{\mathcal{T}_h}}+\langle h^{-1} (P_Mw_h-\mu_h, P_Mw_h-\mu_h \rangle_{\partial{{\mathcal{T}_h}}\backslash\varepsilon_h^{\partial}} + \langle h^{-1} P_Mw_h, P_Mw_h \rangle_{\varepsilon_h^{\partial}}.
	\end{align*}
\end{proof}
\begin{proposition}\label{eq_B}
	There exists a unique solution of the HDG equations \eqref{HDG_full_discrete}.
\end{proposition}
\begin{proof}
	Since the system \eqref{HDG_full_discrete} is finite dimensional, we only need to prove the uniqueness.  Therefore, we assume $y_d = f = 0$ and we show the system \eqref{HDG_full_discrete} only has the trivial solution.
	
	First, by the definition of $ \mathscr B $, we have
	\begin{align*}
	\hspace{1em}&\hspace{-1em}\mathscr B(\bm q_h,y_h,\widehat y_h^o;\bm p_h,-z_h,-\widehat z_h^o) + \mathscr B (\bm p_h,z_h,\widehat z_h^o;-\bm q_h,y_h,\widehat y_h^o)\\
	&=(\bm{q}_h, \bm p_h)_{{\mathcal{T}_h}}- (y_h, \nabla\cdot \bm p_h)_{{\mathcal{T}_h}}+\langle \widehat{y}_h^o, \bm p_h\cdot \bm{n} \rangle_{\partial{{\mathcal{T}_h}}\backslash {\varepsilon_h^{\partial}}} + (\bm{q}_h, \nabla z_h)_{{\mathcal{T}_h}}\\
	& \quad - \langle\bm q_h\cdot\bm n +h^{-1} P_My_h , z_h \rangle_{\partial{{\mathcal{T}_h}}} + \langle h^{-1} \widehat y_h^o, z_h \rangle_{\partial{{\mathcal{T}_h}}\backslash \varepsilon_h^{\partial}}\\
	& \quad + \langle\bm q_h\cdot\bm n + h^{-1} (P_My_h-\widehat y_h^o), \widehat z_h^o  \rangle_{\partial{{\mathcal{T}_h}}\backslash\varepsilon_h^{\partial}}-(\bm{p}_h, \bm q_h)_{{\mathcal{T}_h}}+ (z_h, \nabla\cdot \bm q_h)_{{\mathcal{T}_h}}\\
	&\quad -\langle \widehat{z}_h^o, \bm q_h \cdot \bm{n} \rangle_{\partial{{\mathcal{T}_h}}\backslash {\varepsilon_h^{\partial}}}  - (\bm{p}_h, \nabla y_h)_{{\mathcal{T}_h}}+\langle\bm p_h\cdot\bm n +h^{-1} P_Mz_h , y_h \rangle_{\partial{{\mathcal{T}_h}}}\\
	& \quad - \langle h^{-1} \widehat z_h^o, y_h  \rangle_{\partial{{\mathcal{T}_h}}\backslash \varepsilon_h^{\partial}} - \langle\bm p_h\cdot\bm n + h^{-1} (P_Mz_h-\widehat z_h^o), \widehat y_h^o \rangle_{\partial{{\mathcal{T}_h}}\backslash\varepsilon_h^{\partial}}.
	\end{align*}
	Integrating by parts and using the properties of $P_M$ in \eqref{L2_projection} gives
	\begin{align*}
	\mathscr B & (\bm q_h,y_h,\widehat y_h^o;\bm p_h,-z_h,-\widehat z_h^o) + \mathscr B (\bm p_h,z_h,\widehat z_h^o;-\bm q_h,y_h,\widehat y_h^o) = 0.
	\end{align*}
	
	Next, take $(\bm r_1,w_1,\mu_1) = (\bm p_h,-z_h,-\widehat z_h^o)$, $(\bm r_2,w_2,\mu_2) = (-\bm q_h,y_h,\widehat y_h^o)$,  and $\mu_3 = -\gamma u_h$ in the HDG equations \eqref{HDG_full_discrete_a},  \eqref{HDG_full_discrete_b}, and \eqref{HDG_full_discrete_e}, respectively, and sum to obtain
	\begin{align*}
	(y_h,y_h)_{\mathcal T_h} + \gamma \norm {u_h}^2_{\varepsilon_h^\partial} = 0.
	\end{align*}
	This implies $y_h = 0$ and $u_h = 0$ since $\gamma>0$.
	
	Next, taking $(\bm r_1,w_1,\mu_1) = (\bm q_h,y_h,\widehat y_h^o)$ and $(\bm r_2,w_2,\mu_2) = (\bm p_h,z_h,\widehat z_h^o)$ in Lemma \ref{energy_norm} gives $\bm q_h= \bm p_h= \bm 0, \widehat y_h^o =0$,  $P_M z_h = 0$ on $\varepsilon_h^{\partial}$, and $P_M z_h - \widehat z_h^o=0$ on $\partial\mathcal T_h\backslash \varepsilon_h^{\partial}$.  Also, since $\widehat z_h=0$ on $\varepsilon_h^{\partial}$ we have
	\begin{align}\label{uniqu}
	P_M z_h-\widehat z_h =0.
	\end{align}
	Substituting \eqref{uniqu} into \eqref{HDG_discrete2_c}, and remembering again $\widehat z_h = 0$ on $\varepsilon_h^\partial$, we get
	\begin{align*}
	- (z_h, \nabla\cdot \bm{r_2})_{{\mathcal{T}_h}}+\langle P_M z_h, \bm{r_2}\cdot \bm{n} \rangle_{\partial{{\mathcal{T}_h}}} =0.
	\end{align*}
	Use the property of $P_M$ in \eqref{L2_projection}, integrate by parts, and take $\bm r_2 = \nabla z_h$ to obtain
	\begin{align*}
	(\nabla z_h, \nabla z_h)_{{\mathcal{T}_h}}=0.
	\end{align*}
	Thus, $z_h$ is  constant on each $K\in\mathcal T_h$, and also $ z_h = P_M z_h = \widehat z_h $ on $\partial\mathcal T_h $.  Since $\widehat z_h=0$ on $\varepsilon_h^{\partial}$ and single valued on each face, we have $z_h=0$ on each $K\in\mathcal T_h$, and therefore also $ \widehat{z}_h^o = 0 $.
\end{proof}

\subsection{Proof of Main Result}

To prove the main result, we follow a similar strategy taken by Gong and Yan \cite{MR2806572}, see also \cite{MR2114385,MR2398768,MR2679576}, and introduce an auxiliary problem with the approximate control $ u_h $ in \eqref{HDG_full_discrete_a} replaced by a projection of the exact optimal control.  We first bound the error between the solutions of the auxiliary problem and the mixed weak form \eqref{boundary_pro_a}-\eqref{boundary_pro_d} of the optimality system.  The we bound the error between the solutions of the auxiliary problem and the HDG problem \eqref{HDG_full_discrete}.  A simple application of the triangle inequality then gives a bound on the error between the solutions of the HDG problem and then mixed form of the optimality system.

The precise form of the auxiliary problem is given as follows: find $({\bm{q}}_h(u),\\{\bm{p}}_h(u),y_h(u),z_h(u),{\widehat{y}}_h^o(u),{\widehat{z}}_h^o(u))\in \bm{V}_h\times\bm{V}_h\times W_h \times W_h\times M_h(o)\times M_h(o)$ such that
\begin{subequations}\label{HDG_u}
	\begin{align}
	\mathscr B (\bm q_h(u),y_h(u),\widehat y_h^o(u);\bm r_1,w_1,\mu_1) &= -\langle P_M u, \bm{r_1}\cdot \bm{n} - h^{-1} w_1 \rangle_{{\varepsilon_h^{\partial}}}+(f, w_1)_{{\mathcal{T}_h}}, \label{HDG_u_a}\\
	\mathscr B(\bm p_h(u),z_h(u),\widehat z_h^o(u);\bm r_2,w_2,\mu_2) &= (y_h(u) -y_d,w_2)_{\mathcal T_h}.\label{HDG_u_b}
	\end{align}
\end{subequations}
for all $\left(\bm{r}_1, \bm{r}_2, w_1, w_2, \mu_1, \mu_2\right)\in \bm V_h\times\bm V_h\times W_h\times W_h\times M_h(o)\times M_h(o)$.

We split the proof of the main result, Theorem \ref{main_res}, in 7 steps.  We begin by bounding the error between the solutions of the auxiliary problem and the mixed form \eqref{boundary_pro_a}-\eqref{boundary_pro_d} of the optimality system.  We split the errors in the variables using the $ L^2 $ projections.  In steps 1-3, we focus on the primary variables, i.e., the state $ y $ and the flux $ \bm{q} $, and we use the following notation:
\begin{equation}\label{notation_1}
\begin{split}
\delta^{\bm q} &=\bm q-{\bm\Pi}\bm q,  \qquad\qquad\qquad \qquad\qquad\qquad\;\;\;\;\varepsilon^{\bm q}_h={\bm\Pi} \bm q-\bm q_h(u),\\
\delta^y&=y- {\Pi} y, \qquad\qquad\qquad \qquad\qquad\qquad\;\;\;\; \;\varepsilon^{y}_h={\Pi} y-y_h(u),\\
\delta^{\widehat y} &= y-P_My,  \qquad\qquad\qquad\qquad\qquad\qquad \;\;\; \varepsilon^{\widehat y}_h=P_M y-\widehat{y}_h(u),\\
\widehat {\bm\delta}_1 &= \delta^{\bm q}\cdot\bm n+h^{-1} P_M \delta^y,  \qquad\qquad\qquad\qquad\quad
\widehat {\bm \varepsilon }_1= \varepsilon_h^{\bm q}\cdot\bm n+h^{-1}(P_M\varepsilon^y_h-\varepsilon_h^{\widehat y}),
\end{split}
\end{equation}
where $\widehat y_h(u) = \widehat y_h^o(u)$ on $\varepsilon_h^o$ and $\widehat y_h(u) = P_M u$ on $\varepsilon_h^{\partial}$.  Note that this implies $\varepsilon_h^{\widehat y} = 0$ on $\varepsilon_h^{\partial}$.

\subsubsection{Step 1: The error equation for part 1 of the auxiliary problem \eqref{HDG_u_a}}
\label{subsec:proof_step1}

\begin{lemma}\label{lemma:step1_first_lemma}
	We have
	\begin{align}\label{error_equation_L2k1}
	\mathscr B(\varepsilon_h^{\bm q},\varepsilon_h^{ y}, \varepsilon_h^{\widehat y}; \bm r_1, w_1, \mu_1) = -\langle \widehat {\bm\delta}_1, w_1\rangle_{\partial\mathcal T_h}+\langle\widehat {\bm\delta}_1 ,\mu_1\rangle_{\partial\mathcal T_h\backslash\varepsilon^{\partial}_h}.
	\end{align}
\end{lemma}

\begin{proof}
	By the definition of the operator $ \mathscr B $ in \eqref{def_B}, we have
	\begin{align*}
	\hspace{2em}&\hspace{-2em}\mathscr B(\bm \Pi {\bm q},\Pi { y}, P_M  y; \bm r_1, w_1, \mu_1)\\
	&= (\bm \Pi {\bm q}, \bm{r_1})_{{\mathcal{T}_h}}- (\Pi { y}, \nabla\cdot \bm{r_1})_{{\mathcal{T}_h}}+\langle P_M  y, \bm{r_1}\cdot \bm{n} \rangle_{\partial{{\mathcal{T}_h}}\backslash {\varepsilon_h^{\partial}}}\\
	& \quad - (\bm \Pi {\bm q}, \nabla w_1)_{{\mathcal{T}_h}}+\langle \bm \Pi {\bm q}\cdot\bm n +h^{-1} P_M\Pi { y} , w_1 \rangle_{\partial{{\mathcal{T}_h}}} \\
	& \quad - \langle h^{-1} P_M  y, w_1 \rangle_{\partial{{\mathcal{T}_h}}\backslash \varepsilon_h^{\partial}} - \langle\bm \Pi {\bm q}\cdot\bm n - h^{-1} P_M \delta^y, \mu_1 \rangle_{\partial{{\mathcal{T}_h}}\backslash\varepsilon_h^{\partial}}.
	\end{align*}
	By properties of the $ L^2 $ projections \eqref{L2_projection}, we have
	\begin{align*}
	\mathscr B(\bm \Pi {\bm q},\Pi { y}, P_M  y; \bm r_1, w_1, \mu_1) &= ( {\bm q}, \bm{r_1})_{{\mathcal{T}_h}}- ({ y}, \nabla\cdot \bm{r_1})_{{\mathcal{T}_h}}+\langle   y, \bm{r_1}\cdot \bm{n} \rangle_{\partial{{\mathcal{T}_h}}\backslash {\varepsilon_h^{\partial}}}\\
	& \quad - ( {\bm q}, \nabla w_1)_{{\mathcal{T}_h}}+\langle {\bm q}\cdot\bm n, w_1 \rangle_{\partial{{\mathcal{T}_h}}} - \langle \delta^{\bm q}\cdot\bm n, w_1 \rangle_{\partial{{\mathcal{T}_h}}}\\
	& \quad +\langle h^{-1} P_M\Pi { y} , w_1 \rangle_{\partial{{\mathcal{T}_h}}} - \langle h^{-1} P_M  y, w_1 \rangle_{\partial{{\mathcal{T}_h}}\backslash \varepsilon_h^{\partial}}\\
	& \quad - \langle {\bm q}\cdot\bm n, \mu_1 \rangle_{\partial{{\mathcal{T}_h}}\backslash\varepsilon_h^{\partial}} + \langle \delta^{\bm q}\cdot\bm n, \mu_1 \rangle_{\partial{{\mathcal{T}_h}}\backslash\varepsilon_h^{\partial}}\\
	& \quad +\langle h^{-1} P_M \delta^y, \mu_1 \rangle_{\partial{{\mathcal{T}_h}}\backslash\varepsilon_h^{\partial}}.
	\end{align*}
	Note that the exact state $ y $ and exact flux $\bm{q}$ satisfy
	\begin{align*}
	(\bm{q},\bm{r}_1)_{\mathcal{T}_h}-(y,\nabla\cdot \bm{r}_1)_{\mathcal{T}_h}+\left\langle{y},\bm r_1\cdot \bm n \right\rangle_{\partial {\mathcal{T}_h}\backslash \varepsilon_h^{\partial}} &= -\left\langle u,\bm r_1\cdot \bm n \right\rangle_{\varepsilon_h^{\partial}},\\
	-(\bm{q},\nabla w_1)_{\mathcal{T}_h}+\left\langle {\bm{q}}\cdot \bm{n},w_1\right\rangle_{\partial {\mathcal{T}_h}} &= (f,w_1)_{\mathcal{T}_h},\\
	\left\langle {\bm{q}}\cdot \bm{n},\mu_1\right\rangle_{\partial {\mathcal{T}_h}\backslash \varepsilon_h^{\partial}}&=0,
	\end{align*}
	for all $(\bm{r}_1,w_1,\mu_1)\in\bm{V}_h\times W_h\times M_h(o)$. Then we have
	\begin{align*}
	\mathscr B(\bm \Pi {\bm q},\Pi { y}, P_M  y; \bm r_1, w_1, \mu_1)=&-\left\langle u,\bm r_1\cdot \bm n \right\rangle_{\varepsilon_h^{\partial}} + (f,w_1)_{\mathcal T_h} - \langle \delta^{\bm q}\cdot\bm n, w_1 \rangle_{\partial{{\mathcal{T}_h}}}\\
	&+\langle h^{-1} P_M\Pi { y} , w_1 \rangle_{\partial{{\mathcal{T}_h}}} - \langle h^{-1} P_M  y, w_1 \rangle_{\partial{{\mathcal{T}_h}}\backslash \varepsilon_h^{\partial}}\\
	&+ \langle \delta^{\bm q}\cdot\bm n, \mu_1 \rangle_{\partial{{\mathcal{T}_h}}\backslash\varepsilon_h^{\partial}}+ \langle h^{-1} P_M\delta^y, \mu_1 \rangle_{\partial{{\mathcal{T}_h}}\backslash\varepsilon_h^{\partial}}.
	\end{align*}
	Subtract part 1 of the auxiliary problem \eqref{HDG_u_a} from the above equality to obtain the result:
	\begin{align*}
	\mathscr B(\varepsilon_h^{\bm q},\varepsilon_h^{ y}, \varepsilon_h^{\widehat y}; \bm r_1, w_1, \mu_1)  = &-\langle P_M u, h^{-1} w_1 \rangle_{{\varepsilon_h^{\partial}}} - \langle \delta^{\bm q}\cdot\bm n, w_1 \rangle_{\partial{{\mathcal{T}_h}}}\\
	&+\langle h^{-1} P_M\Pi { y} , w_1 \rangle_{\partial{{\mathcal{T}_h}}} - \langle h^{-1} P_M  y, w_1 \rangle_{\partial{{\mathcal{T}_h}}\backslash \varepsilon_h^{\partial}}\\
	&+ \langle \delta^{\bm q}\cdot\bm n, \mu_1 \rangle_{\partial{{\mathcal{T}_h}}\backslash\varepsilon_h^{\partial}}+ \langle h^{-1} P_M\delta^y, \mu_1 \rangle_{\partial{{\mathcal{T}_h}}\backslash\varepsilon_h^{\partial}}\\
	= &  -\langle \widehat {\bm\delta}_1, w_1\rangle_{\partial\mathcal T_h}+\langle\widehat {\bm\delta}_1 ,\mu_1\rangle_{\partial\mathcal T_h\backslash\varepsilon^{\partial}_h}.
	\end{align*}
\end{proof}

\subsubsection{Step 2: Estimate for $\varepsilon_h^{\boldmath q}$}  We first provide a key inequality which was proven in \cite{MR3440284}.
\begin{lemma}\label{nabla_ine}
	We have
	\begin{equation*}
	\|\nabla\varepsilon^y_h\|_{\mathcal T_h}+ h^{-\frac{1}{2}} \| {\varepsilon_h^y -\varepsilon_h^{\widehat y}}\|_{\partial \mathcal T_h}
	\lesssim \|\varepsilon^{\bm q}_h\|_{\mathcal T_h}+h^{-\frac1 2}\|P_M\varepsilon^y_h-\varepsilon^{\widehat y}_h\|_{\partial\mathcal T_h}.
	\end{equation*}
\end{lemma}

\begin{lemma}\label{lemma:step2_main_lemma}
	We have
	\begin{align}
	\norm{\varepsilon_h^{\bm{q}}}_{\mathcal{T}_h}^2+h^{-1}\|{P_M\varepsilon_h^y-\varepsilon_h^{\widehat{y}}}\|_{\partial \mathcal T_h}^2 \lesssim h^{2s_{\bm q}}\norm{\bm q}_{s^{\bm q},\Omega}^2 + h^{2s_{y}-2}\norm{y}_{s^{y},\Omega}^2.
	\end{align}
\end{lemma}
\begin{proof}
	First, since $\varepsilon_h^{\widehat y}=0$ on $\varepsilon_h^\partial$, the basic property of $ \mathscr B $ in Lemma \ref{energy_norm} gives
	\[ \mathscr B(\varepsilon_h^{\bm q},\varepsilon_h^{ y}, \varepsilon_h^{\widehat y}; \varepsilon_h^{\bm q},\varepsilon_h^{ y}, \varepsilon_h^{\widehat y}) =(\varepsilon_h^{\bm{q}},\varepsilon_h^{\bm{q}})_{\mathcal{T}_h}+h^{-1} \|{P_M\varepsilon_h^y-\varepsilon_h^{\widehat{y}}}\|_{\partial \mathcal T_h}^2. \]
	Then, taking $(\bm r_1, w_1,\mu_1) = (\bm \varepsilon_h^{\bm q},\varepsilon_h^y,\varepsilon_h^{\widehat y})$ in \eqref{error_equation_L2k1} in Lemma \ref{lemma:step1_first_lemma} gives
	\begin{align*}
	\hspace{4em}&\hspace{-4em}(\varepsilon_h^{\bm{q}},\varepsilon_h^{\bm{q}})_{\mathcal{T}_h}+h^{-1} \|{P_M\varepsilon_h^y-\varepsilon_h^{\widehat{y}}}\|_{\partial \mathcal T_h}^2\\
	&= -\langle \widehat {\bm\delta}_1,\varepsilon_h^y - \varepsilon_h^{\widehat y}\rangle_{\partial\mathcal T_h}\\
	&= -\langle\delta^{\bm q}\cdot \bm n,\varepsilon_h^y - \varepsilon_h^{\widehat y} \rangle_{\partial\mathcal T_h}- h^{-1} \langle \delta^y,P_M \varepsilon_h^y - \varepsilon_h^{\widehat y}\rangle_{\partial\mathcal T_h}\\
	&\le \norm{\delta^{\bm q}}_{\partial \mathcal T_h}\| {\varepsilon_h^y - \varepsilon_h^{\widehat y}}\|_{\partial \mathcal T_h} + h^{-1}
	\|{\delta ^y}\|_{\partial\mathcal T_h} \|{P_M \varepsilon_h^y - \varepsilon_h^{\widehat y}}\|_{\partial\mathcal T_h}\\
	&\le h^{1/2}\norm{\delta^{\bm q}}_{\partial \mathcal T_h}h^{-1/2}\| {\varepsilon_h^y - \varepsilon_h^{\widehat y}}\|_{\partial \mathcal T_h}  +h^{-1}
	\|{\delta ^y}\|_{\partial\mathcal T_h} \|{P_M \varepsilon_h^y - \varepsilon_h^{\widehat y}}\|_{\partial\mathcal T_h}.
	\end{align*}
	By Young's inequality and Lemma \ref{nabla_ine}, we obtain
	\begin{align*}
	\|\varepsilon_h^{\bm{q}}\|_{\mathcal{T}_h}^2+h^{-1}\|{P_M\varepsilon_h^y-\varepsilon_h^{\widehat{y}}}\|_{\partial \mathcal T_h}^2 & \lesssim h\norm{\delta^{\bm q}}_{\partial\mathcal T_h}^2 + h^{-1}
	\norm{\delta^{y}}_{\partial\mathcal T_h}^2\\
	&\lesssim h^{2s_{\bm q}}\norm{\bm q}_{s^{\bm q},\Omega}^2 + h^{2s_{y}-2}\norm{y}_{s^{y},\Omega}^2.
	\end{align*}
\end{proof}
%
%

\subsubsection{Step 3: Estimate for $\varepsilon_h^{y}$ by a duality argument}  Next, we introduce the dual problem for any given $\Theta$ in $L^2(\Omega):$
\begin{equation}\label{Dual_PDE}
\begin{split}
\bm{\Phi}+\nabla\Psi &= 0\qquad\qquad\text{in}\ \Omega,\\
\nabla\cdot\bm \Phi &= \Theta\qquad\quad~~\text{in}\ \Omega,\\
\Psi &= 0\qquad\qquad\text{on}\ \Gamma.
\end{split}
\end{equation}
Since the domain $\Omega$ is convex, we have the following regularity estimate
\begin{align}\label{reg_e}
\norm{\bm\Phi}_{H^1(\Omega)} + \norm{\Psi}_{H^2(\Omega)} \le C \norm{\Theta}_\Omega.
\end{align}

Before we estimate $\varepsilon_h^{y}$ we introduce the following notation, which is similar to the earlier notation in \eqref{notation_1}:
\begin{align}
\delta^{\bm \Phi} &=\bm \Phi-{\bm\Pi}\bm \Phi, \quad \delta^\Psi=\Psi- {\Pi} \Psi, \quad
\delta^{\widehat \Psi} = \Psi-P_M\Psi.
\end{align}
By the regularity estimate \eqref{reg_e}, we have the following bounds:
\begin{align}
\|\delta^{\bm \Phi} \|_{\mathcal T_h} \lesssim h\|\Theta\|_{\mathcal T_h},\quad \|\delta^{ \Psi} \|_{\mathcal T_h} \lesssim h^2\|\Theta\|_{\mathcal T_h},\quad \|\delta^{ \widehat \Psi} \|_{\partial\mathcal T_h} \lesssim h^{\frac 1 2}\|\Theta\|_{\mathcal T_h}.
\end{align}

\begin{lemma}\label{lemma:step3_first_lemma}
	We have
	\begin{align}
	\|\varepsilon_h^y\|_{\mathcal T_h} \lesssim h^{s_{\bm q}+1}\norm{\bm q}_{s^{\bm q},\Omega} + h^{s_{y}}\norm{y}_{s^{y},\Omega}.
	\end{align}
\end{lemma}

\begin{proof}
	Consider the dual problem \eqref{Dual_PDE} and let $\Theta = \varepsilon_h^y$.  In the definition \eqref{def_B} of $ \mathscr B $, take $ (\bm r_1, w_1, \mu_1) $ to be $ (-\bm\Pi\bm \Phi,\Pi\Psi,P_M\Psi) $ and use $\Psi=0$ on $\varepsilon_h^{\partial}$ to obtain
	\begin{align}\label{eqn:dual_equations}
	\mathscr B(\varepsilon_h^{\bm q},\varepsilon_h^{ y}, \varepsilon_h^{\widehat y}; -\bm\Pi\bm \Phi,\Pi\Psi,P_M\Psi) =& -(\varepsilon_h^{\bm q},\bm\Pi\bm \Phi)_{\mathcal T_h}+(\varepsilon^y_h,\nabla\cdot\bm\Pi\bm \Phi)_{\mathcal T_h}-\langle\varepsilon^{\widehat y}_h,\bm\Pi\bm \Phi\cdot\bm n\rangle_{\partial\mathcal T_h}\nonumber\\
	&-(\varepsilon_h^{\bm q},\nabla\Pi\Psi)_{\mathcal T_h}+\langle\widehat{\bm \varepsilon}_1,\Pi\Psi\rangle_{\partial\mathcal T_h} - \langle\widehat{\bm \varepsilon}_1,P_M\Psi\rangle_{\partial\mathcal T_h}.
	\end{align}
	Next, it is easy to verify that
	\begin{align*}
	(\varepsilon^y_h,\nabla\cdot\bm\Pi \bm \Phi)_{\mathcal T_h} &= \langle\varepsilon^y_h,\bm\Pi\bm \Phi\cdot\bm n\rangle_{\partial\mathcal T_h}-(\nabla\varepsilon^y_h,\bm\Pi\bm \Phi)_{\mathcal T_h}\\
	&= \langle\varepsilon^y_h,\bm\Pi\bm \Phi\cdot\bm n\rangle_{\partial\mathcal T_h}-(\nabla\varepsilon^y_h,\bm \Phi)_{\mathcal T_h}\\
	&= -\langle\varepsilon^y_h,\delta^{\bm \Phi}\cdot\bm n\rangle_{\partial\mathcal T_h}+(\varepsilon^y_h,\nabla\cdot\bm \Phi)_{\mathcal T_h}\\
	&= -\langle\varepsilon^y_h,\delta^{\bm \Phi}\cdot\bm n\rangle_{\partial\mathcal T_h}+\|\varepsilon^y_h\|^2_{\mathcal T_h}.
	\end{align*}
	Similarly,
	\begin{align*}
	-(\varepsilon_h^{\bm q},\nabla\Pi\Psi)_{\mathcal T_h}
	&= -\langle\varepsilon_h^{\bm q}\cdot \bm n,\Pi \Psi\bm \rangle_{\partial\mathcal T_h}+(\nabla\cdot\varepsilon_h^{\bm q},\Pi\Psi)_{\mathcal T_h}\\
	&= -\langle\varepsilon_h^{\bm q}\cdot \bm n,\Pi \Psi\bm \rangle_{\partial\mathcal T_h}+(\nabla\cdot\varepsilon_h^{\bm q},\Psi)_{\mathcal T_h}\\
	&= -\langle\varepsilon_h^{\bm q}\cdot \bm n,\Pi \Psi\bm \rangle_{\partial\mathcal T_h} +\langle\varepsilon_h^{\bm q}\cdot \bm n,\Psi\bm \rangle_{\partial\mathcal T_h}-(\varepsilon_h^{\bm q},\nabla\Psi)_{\mathcal T_h} \\
	&= \langle\varepsilon_h^{\bm q}\cdot\bm n,(P_M\Psi-\Pi \Psi) \rangle_{\partial\mathcal T_h}-(\varepsilon_h^{\bm q},\nabla\Psi)_{\mathcal T_h}.
	\end{align*}
	Then equation \eqref{eqn:dual_equations} becomes
	\begin{align*}
	\hspace{2em}&\hspace{-2em} \mathscr B(\varepsilon_h^{\bm q},\varepsilon_h^{ y}, \varepsilon_h^{\widehat y}; -\bm\Pi\bm \Phi,\Pi\Psi,P_M\Psi)\\
	&= -(\varepsilon_h^{\bm q},\bm\Phi)_{\mathcal T_h}-\langle\varepsilon_h^y,\delta^{\bm \Phi}\cdot\bm n\rangle_{\partial\mathcal T_h}+\|\varepsilon_h^y\|^2_{\mathcal T_h}-\langle\varepsilon^{\widehat y}_h,\bm\Pi\bm \Phi\cdot\bm n\rangle_{\partial\mathcal T_h}\\
	&+\langle\varepsilon_h^{\bm q}\cdot \bm n, P_M\Psi-\Pi \Psi \rangle_{\partial\mathcal T_h}-(\varepsilon_h^{\bm q},\nabla\Psi)_{\mathcal T_h}+\langle\widehat{\bm \varepsilon}_1,\Pi \Psi\rangle_{\partial\mathcal T_h} - \langle\widehat{\bm \varepsilon}_1, P_M\Psi\rangle_{\partial\mathcal T_h}.
	\end{align*}
	The facts $\bm \Phi+\nabla\Psi=0$, $\langle\varepsilon^{\widehat y}_h,\bm \Phi\cdot\bm n\rangle_{\partial\mathcal T_h}=0$, and  $\langle\widehat{\bm \varepsilon}_1, P_M\Psi\rangle_{\partial\mathcal T_h} = \langle\widehat{\bm \varepsilon}_1, \Psi\rangle_{\partial\mathcal T_h}$ imply
	\begin{align*}
	\hspace{2em}&\hspace{-2em} \mathscr B(\varepsilon_h^{\bm q},\varepsilon_h^{ y}, \varepsilon_h^{\widehat y}; -\bm\Pi\bm \Phi,\Pi\Psi,P_M\Psi)\\
	&= -\langle\varepsilon_h^y - \varepsilon^{\widehat y}_h,\delta^{\bm \Phi}\cdot\bm n\rangle_{\partial\mathcal T_h}+\|\varepsilon_h^y\|^2_{\mathcal T_h} - h^{-1}\langle P_M\varepsilon^y_h-\varepsilon^{\widehat y}_h, \delta^{ \Psi} \rangle_{\partial\mathcal T_h}.
	\end{align*}
	On the other hand,  equation \eqref{error_equation_L2k1} in  Lemma \ref{lemma:step1_first_lemma} gives
	\begin{align*}
	\mathscr B(\varepsilon_h^{\bm q},\varepsilon_h^{ y}, \varepsilon_h^{\widehat y}; -\bm\Pi\bm \Phi,\Pi\Psi,P_M\Psi) =&-\langle \widehat {\bm\delta}_1, \Pi\Psi\rangle_{\partial\mathcal T_h}+\langle\widehat {\bm\delta}_1 ,P_M\Psi\rangle_{\partial\mathcal T_h\backslash\varepsilon^{\partial}_h}.
	\end{align*}
	Moreover,
	\begin{align*}
	\hspace{2em}&\hspace{-2em} \langle\widehat {\bm\delta}_1 ,P_M\Psi\rangle_{\partial\mathcal T_h\backslash\varepsilon^{\partial}_h}\\
	&= \langle \delta^{\bm q}\cdot\bm n+h^{-1} P_M \delta^y ,P_M\Psi\rangle_{\partial\mathcal T_h\backslash\varepsilon^{\partial}_h} \\
	&= \langle \bm q \cdot\bm n ,P_M\Psi\rangle_{\partial\mathcal T_h\backslash\varepsilon^{\partial}_h}
	- \langle \bm \Pi  \bm q\cdot\bm n ,P_M\Psi\rangle_{\partial\mathcal T_h\backslash\varepsilon^{\partial}_h}
	+\langle h^{-1} P_M \delta^y, P_M\Psi\rangle_{\partial\mathcal T_h\backslash\varepsilon^{\partial}_h} \\
	&= - \langle \bm \Pi  \bm q\cdot\bm n ,\Psi\rangle_{\partial\mathcal T_h\backslash\varepsilon^{\partial}_h}
	+\langle h^{-1} P_M \delta^y, \Psi\rangle_{\partial\mathcal T_h\backslash\varepsilon^{\partial}_h} \\
	&= \langle \bm q \cdot\bm n, \Psi\rangle_{\partial\mathcal T_h\backslash\varepsilon^{\partial}_h} - \langle \bm \Pi  \bm q\cdot\bm n ,\Psi\rangle_{\partial\mathcal T_h\backslash\varepsilon^{\partial}_h}
	+\langle h^{-1} P_M \delta^y, \Psi\rangle_{\partial\mathcal T_h\backslash\varepsilon^{\partial}_h} \\
	&= \langle\widehat {\bm\delta}_1 ,\Psi\rangle_{\partial\mathcal T_h\backslash\varepsilon^{\partial}_h}\\
	&= \langle\widehat {\bm\delta}_1 ,\Psi\rangle_{\partial\mathcal T_h},
	\end{align*}
	where we have used $ \langle \bm q \cdot \bm n,P_M \Psi\rangle_{\partial\mathcal T_h\backslash\varepsilon_h^\partial} = 0 $, $ \langle \bm q \cdot \bm n, \Psi\rangle_{\partial\mathcal T_h\backslash\varepsilon_h^\partial} =0 $ since $\bm q \in H(\text{div}, \Omega)$ and $\Psi = 0$ on $\varepsilon_h^\partial$.
	
	Comparing the above two equalities gives
	\begin{align*}
	\|\varepsilon^y_h\|^2_{\mathcal T_h}&=\langle\varepsilon^y_h-\varepsilon^{\widehat y}_h,\delta^{\bm \Phi}\cdot\bm n\rangle_{\partial\mathcal T_h}+h^{-1}\langle P_M\varepsilon^y_h-\varepsilon^{\widehat y}_h, \delta^{ \Psi} \rangle_{\partial\mathcal T_h}+\langle\widehat{\bm \delta}_1, \delta^{ \Psi} \rangle_{\partial\mathcal T_h}\\
	&=\langle\varepsilon^y_h-\varepsilon^{\widehat y}_h,\delta^{\bm \Phi}\cdot\bm n\rangle_{\partial\mathcal T_h}+h^{-1}\langle P_M\varepsilon^y_h-\varepsilon^{\widehat y}_h,\delta^{ \Psi}\rangle_{\partial\mathcal T_h}\\
	& \quad -\langle \delta^{\bm q}\cdot\bm n+h^{-1} P_M\delta^{ y},\delta^{ \Psi}\rangle_{\partial\mathcal T_h}\\
	&\lesssim h^{-\frac12}\|\varepsilon^y_h-\varepsilon^{\widehat y}_h\|_{\partial\mathcal T_h}\cdot h^{\frac 1 2}\|\delta^{ \bm \Phi}\|_{\partial\mathcal T_h}+h^{-\frac12}\|P_M\varepsilon^y_h-\varepsilon^{\widehat y}_h\|_{\partial\mathcal T_h}\cdot h^{-\frac12}\|\delta^{ \Psi}\|_{\partial\mathcal T_h}\\
	& \quad +\|\delta^{ \bm q}\|_{\partial\mathcal T_h}\cdot\|\delta^{ \Psi} \|_{\partial\mathcal T_h}+h^{-1}\|\delta^{y}\|_{\partial\mathcal T_h}\cdot\| \delta^{ \Psi} \|_{\partial\mathcal T_h}\\
	&\lesssim (h^{s_{\bm q}+1}\norm{\bm q}_{s^{\bm q},\Omega} + h^{s_{y}}\norm{y}_{s^{y},\Omega}) \|\varepsilon^y_h\|_{\mathcal T_h}.
	\end{align*}
\end{proof}

As a consequence of Lemma \ref{lemma:step2_main_lemma} and Lemma \ref{lemma:step3_first_lemma}, a simple application of the triangle inequality gives optimal convergence rates for $\|\bm q -\bm q_h(u)\|_{\mathcal T_h}$ and $\|y -y_h(u)\|_{\mathcal T_h}$:
\begin{lemma}\label{lemma:step3_conv_rates}
	\begin{subequations}
		\begin{align}
		\|\bm q -\bm q_h(u)\|_{\mathcal T_h}&\le \|\delta^{\bm q}\|_{\mathcal T_h} + \|\varepsilon_h^{\bm q}\|_{\mathcal T_h} \lesssim h^{s_{\bm q}}\norm{\bm q}_{s^{\bm q},\Omega} + h^{s_{y}-1}\norm{y}_{s^{y},\Omega},\\
		\|y -y_h(u)\|_{\mathcal T_h}&\le \|\delta^{y}\|_{\mathcal T_h} + \|\varepsilon_h^{y}\|_{\mathcal T_h} \lesssim h^{s_{\bm q}+1}\norm{\bm q}_{s^{\bm q},\Omega} + h^{s_{y}}\norm{y}_{s^{y},\Omega}.
		\end{align}
	\end{subequations}
\end{lemma}

%
\subsubsection{Step 4: The error equation for part 2 of the auxiliary problem \eqref{HDG_u_b}}  We continue to bound the error between the solutions of the auxiliary problem and the mixed form \eqref{boundary_pro_a}-\eqref{boundary_pro_d} of the optimality system.  In steps 4-5, we focus on the dual variables, i.e., the state $ z $ and the flux $ \bm{p} $.  We split the errors in the variables using the $ L^2 $ projections, and we use the following notation.
\begin{equation}\label{notation_3}
\begin{split}
\delta^{\bm p} &=\bm p-{\bm\Pi}\bm p,  \qquad\qquad\qquad \qquad\qquad\qquad\;\;\;\;\varepsilon^{\bm p}_h={\bm\Pi} \bm p-\bm p_h(u),\\
\delta^z&=z- {\Pi} z, \qquad\qquad\qquad \qquad\qquad\qquad\;\;\;\; \;\varepsilon^{z}_h={\Pi} z-z_h(u),\\
\delta^{\widehat z} &= z-P_Mz,  \qquad\qquad\qquad\qquad\qquad\qquad \;\;\; \varepsilon^{\widehat z}_h=P_M z-\widehat{z}_h(u),\\
\widehat {\bm\delta}_2 &= \delta^{\bm p}\cdot\bm n+h^{-1} P_M \delta^z.
\end{split}
\end{equation}
where $\widehat z_h(u) = \widehat z_h^o(u)$ on $\varepsilon_h^o$ and $\widehat z_h(u) = 0$ on $\varepsilon_h^{\partial}$.  Note that this implies $\varepsilon_h^{\widehat z} = 0$ on $\varepsilon_h^{\partial}$.

The derivation of the error equation for part 2 of the auxiliary problem \eqref{HDG_u_b} is similar to the analysis for part 1 of the auxiliary problem in step 1 in \ref{subsec:proof_step1}; the only difference is there is one more term $(y-y_h(u),w_2)_{\mathcal T_h}$ in the right hand side.  Therefore, we state the result and omit the proof.
\begin{lemma}\label{lemma:step4_first_lemma}
	We have
	\begin{align}\label{error_z}
	\mathscr B(\varepsilon_h^{\bm p},\varepsilon_h^{z}, \varepsilon_h^{\widehat z}, \bm r_2, w_2, \mu_2) = -\langle \widehat {\bm\delta}_2, w_2\rangle_{\partial\mathcal T_h}+\langle\widehat {\bm\delta}_2 ,\mu_2\rangle_{\partial\mathcal T_h\backslash\varepsilon^{\partial}_h}+(y-y_h(u),w_2)_{\mathcal T_h}.
	\end{align}
\end{lemma}

%
\subsubsection{Step 5: Estimate for $\varepsilon_h^{\boldmath p}$ and $\varepsilon_h^{z}$}
Before we estimate  $\varepsilon_h^{\bm p}$, we give the following discrete Poincar{\'e} inequality from \cite{MR3440284}.
\begin{lemma}\label{lemma:discr_Poincare_ineq}
	Since $\varepsilon_h^{\widehat z} = 0$ on $\varepsilon_h^\partial$, we have
	\begin{align}
	\|\varepsilon_h^z\|_{\mathcal T_h} \lesssim \|\nabla \varepsilon_h^z\|_{\mathcal T_h} + h^{-\frac 1 2} \|\varepsilon_h^z - \varepsilon_h^{\widehat z}\|_{\partial\mathcal T_h}.
	\end{align}
\end{lemma}

\begin{lemma}\label{lemma:step5_main_lemma}
	We have
	\begin{align*}
	\norm{\varepsilon_h^{\bm{p}}}_{\mathcal{T}_h}&+h^{-\frac 1 2}\|{P_M\varepsilon_h^z-\varepsilon_h^{\widehat{z}}}\|_{\partial \mathcal T_h}\\
	& \lesssim h^{s_{\bm p}}\norm{\bm p}_{s^{\bm p},\Omega} +  h^{s_{z}-1}\norm{z}_{s^{z},\Omega} + h^{s_{\bm q}+1}\norm{\bm q}_{s^{\bm q},\Omega} + h^{s_{y}}\norm{y}_{s^{y},\Omega},\\
	\|\varepsilon_h^z\|_{\mathcal T_h} & \lesssim h^{s_{\bm p}}\norm{\bm p}_{s^{\bm p},\Omega} +  h^{s_{z}-1}\norm{z}_{s^{z},\Omega} + h^{s_{\bm q}+1}\norm{\bm q}_{s^{\bm q},\Omega} + h^{s_{y}}\norm{y}_{s^{y},\Omega}.
	\end{align*}
\end{lemma}
\begin{proof}
	First, we note the key inequality in Lemma \ref{nabla_ine} is valid with $ (z,\bm p, \hat z) $ in place of $ (y,\bm q, \hat y) $.  This gives
	\[ \|\nabla \varepsilon_h^z\|_{\mathcal T_h} + h^{-\frac 1 2} \|\varepsilon_h^z - \varepsilon_h^{\widehat z}\|_{\partial\mathcal T_h} \lesssim \|\varepsilon^{\bm p}_h\|_{\mathcal T_h}+h^{-\frac1 2}\|P_M\varepsilon^z_h-\varepsilon^{\widehat z}_h\|_{\partial\mathcal T_h}, \]
	which we use below.  Next, since $\varepsilon_h^{\widehat z}=0$ on $\varepsilon_h^\partial$, the basic property of $ \mathscr B $ in Lemma \ref{energy_norm} gives
	\[ \mathscr B(\varepsilon_h^{\bm p},\varepsilon_h^{ z}, \varepsilon_h^{\widehat z}, \varepsilon_h^{\bm p},\varepsilon_h^{z}, \varepsilon_h^{\widehat z}) =(\varepsilon_h^{\bm{p}},\varepsilon_h^{\bm{p}})_{\mathcal{T}_h}+h^{-1} \|{P_M\varepsilon_h^z-\varepsilon_h^{\widehat{z}}}\|_{\partial \mathcal T_h}^2. \]
	Then taking $(\bm r_2, w_2,\mu_2) = (\bm \varepsilon_h^{\bm p},\varepsilon_h^z,\varepsilon_h^{\widehat z})$ in \eqref{error_z} in Lemma \ref{lemma:step4_first_lemma} gives
	%
	%
	\begin{align*}
	\hspace{2em}&\hspace{-2em}(\varepsilon_h^{\bm{p}},\varepsilon_h^{\bm{p}})_{\mathcal{T}_h}+h^{-1} \|{P_M\varepsilon_h^z-\varepsilon_h^{\widehat{z}}}\|_{\partial \mathcal T_h}^2\\
	&= -\langle \widehat {\bm\delta}_2,\varepsilon_h^z - \varepsilon_h^{\widehat z}\rangle_{\partial\mathcal T_h} + (y-y_h(u),\varepsilon_h^z)_{\mathcal T_h}\\
	&= -\langle\delta^{\bm p}\cdot \bm n,\varepsilon_h^z - \varepsilon_h^{\widehat z} \rangle_{\partial\mathcal T_h}- h^{-1} \langle \delta^z,P_M \varepsilon_h^z - \varepsilon_h^{\widehat z}\rangle_{\partial\mathcal T_h}+ (y-y_h(u),\varepsilon_h^z)_{\mathcal T_h}\\
	&\le \norm{\delta^{\bm p}}_{\partial \mathcal T_h}\| {\varepsilon_h^z - \varepsilon_h^{\widehat z}}\|_{\partial \mathcal T_h} + h^{-1}
	\|{\delta ^z}\|_{\partial\mathcal T_h} \norm{P_M \varepsilon_h^z - \varepsilon_h^{\widehat z}}_{\partial\mathcal T_h}\\
	& \quad + \|y-y_h(u)\|_{\mathcal T_h} \|\varepsilon_h^z\|_{\mathcal T_h}\\
	&\le h^{1/2}\norm{\delta^{\bm p}}_{\partial \mathcal T_h}h^{-1/2}\| {\varepsilon_h^z - \varepsilon_h^{\widehat z}}\|_{\partial \mathcal T_h}  +h^{-\frac 1 2}
	\|{\delta ^z}\|_{\partial\mathcal T_h} h^{-\frac 1 2} \norm{P_M \varepsilon_h^z - \varepsilon_h^{\widehat z}}_{\partial\mathcal T_h}\\
	&\quad + \|y-y_h(u)\|_{\mathcal T_h} \|\varepsilon_h^z\|_{\mathcal T_h}\\
	&\le h^{1/2}\norm{\delta^{\bm p}}_{\partial \mathcal T_h} (\|\varepsilon^{\bm p}_h\|_{\mathcal T_h}+h^{-\frac1 2}\|P_M\varepsilon^z_h-\varepsilon^{\widehat z}_h\|_{\partial\mathcal T_h})\\
	& \quad +h^{-\frac 1 2}
	\|{\delta ^z}\|_{\partial\mathcal T_h} h^{-\frac 1 2} \norm{P_M \varepsilon_h^z - \varepsilon_h^{\widehat z}}_{\partial\mathcal T_h}\\
	&\quad + C\|y-y_h(u)\|_{\mathcal T_h} (\|\nabla \varepsilon_h^z\|_{\mathcal T_h} + h^{-\frac 1 2} \|\varepsilon_h^z - \varepsilon_h^{\widehat z}\|_{\partial\mathcal T_h})\\
	&\le  h^{1/2}\norm{\delta^{\bm p}}_{\partial \mathcal T_h} (\|\varepsilon^{\bm p}_h\|_{\mathcal T_h}+h^{-\frac1 2}\|P_M\varepsilon^z_h-\varepsilon^{\widehat z}_h\|_{\partial\mathcal T_h})\\
	& \quad +h^{-\frac 1 2}
	\|{\delta ^z}\|_{\partial\mathcal T_h} h^{-\frac 1 2} \norm{P_M \varepsilon_h^z - \varepsilon_h^{\widehat z}}_{\partial\mathcal T_h}\\
	&\quad + C\|y-y_h(u)\|_{\mathcal T_h} (\|\varepsilon^{\bm p}_h\|_{\mathcal T_h}+h^{-\frac1 2}\|P_M\varepsilon^z_h-\varepsilon^{\widehat z}_h\|_{\partial\mathcal T_h}).
	\end{align*}
	Applying Young's inequality and Lemma \ref{lemma:step3_conv_rates} gives
	\begin{align*}
	\hspace{4em}&\hspace{-4em}(\varepsilon_h^{\bm{p}},\varepsilon_h^{\bm{p}})_{\mathcal{T}_h}+h^{-1} \|{P_M\varepsilon_h^z-\varepsilon_h^{\widehat{z}}}\|_{\partial \mathcal T_h}^2\\
	&\lesssim h \norm{\delta^{\bm p}}_{\partial \mathcal T_h}^2 + h^{-1}\|{\delta ^z}\|_{\partial\mathcal T_h}^2  +\|y_h(u)- y\|^2_{\mathcal T_h}\\
	&\lesssim h^{2s_{\bm p}}\norm{\bm p}_{s^{\bm p},\Omega}^2 +  h^{2s_{z}-2}\norm{z}_{s^{z},\Omega}^2 + h^{2s_{\bm q}+2}\norm{\bm q}_{s^{\bm q},\Omega}^2 + h^{2s_{y}}\norm{y}_{s^{y},\Omega}^2.
	\end{align*}
	This gives
	\begin{align*}
	\|\varepsilon_h^{\bm p}\|_{\mathcal T_h} &+ h^{-\frac1 2}\|P_M\varepsilon^z_h-\varepsilon^{\widehat z}_h\|_{\partial\mathcal T_h}\\
	& \lesssim h^{s_{\bm p}}\norm{\bm p}_{s^{\bm p},\Omega} +  h^{s_{z}-1}\norm{z}_{s^{z},\Omega} + h^{s_{\bm q}+1}\norm{\bm q}_{s^{\bm q},\Omega} + h^{s_{y}}\norm{y}_{s^{y},\Omega},\\
	\|\varepsilon_h^z\|_{\mathcal T_h} & \lesssim \|\nabla \varepsilon_h^z\|_{\mathcal T_h} + h^{-\frac 1 2} \|\varepsilon_h^z - \varepsilon_h^{\widehat z}\|_{\partial\mathcal T_h}\\
	&\lesssim \|\varepsilon^{\bm p}_h\|_{\mathcal T_h}+h^{-\frac1 2}\|P_M\varepsilon^z_h-\varepsilon^{\widehat z}_h\|_{\partial\mathcal T_h}\\
	&\lesssim h^{s_{\bm p}}\norm{\bm p}_{s^{\bm p},\Omega} +  h^{s_{z}-1}\norm{z}_{s^{z},\Omega} + h^{s_{\bm q}+1}\norm{\bm q}_{s^{\bm q},\Omega} + h^{s_{y}}\norm{y}_{s^{y},\Omega}.
	\end{align*}
\end{proof}

As a consequence, a simple application of the triangle inequality gives optimal convergence rates for $\|\bm p -\bm p_h(u)\|_{\mathcal T_h}$ and $\|z -z_h(u)\|_{\mathcal T_h}$:

\begin{lemma}\label{lemma:step5_conv_rates}
	\begin{subequations}
		\begin{align}
		\|\bm p -\bm p_h(u)\|_{\mathcal T_h}&\le \|\delta^{\bm p}\|_{\mathcal T_h} + \|\varepsilon_h^{\bm p}\|_{\mathcal T_h}\nonumber\\
		& \lesssim h^{s_{\bm p}}\norm{\bm p}_{s^{\bm p},\Omega} +  h^{s_{z}-1}\norm{z}_{s^{z},\Omega} + h^{s_{\bm q}+1}\norm{\bm q}_{s^{\bm q},\Omega} + h^{s_{y}}\norm{y}_{s^{y},\Omega},\\
		\|z -z_h(u)\|_{\mathcal T_h}&\le \|\delta^{z}\|_{\mathcal T_h} + \|\varepsilon_h^{z}\|_{\mathcal T_h} \nonumber\\
		& \lesssim h^{s_{\bm p}}\norm{\bm p}_{s^{\bm p},\Omega} +  h^{s_{z}-1}\norm{z}_{s^{z},\Omega} + h^{s_{\bm q}+1}\norm{\bm q}_{s^{\bm q},\Omega} + h^{s_{y}}\norm{y}_{s^{y},\Omega}.
		\end{align}
	\end{subequations}
\end{lemma}


\subsubsection{Step 6: Estimate for $\|u-u_h\|_{\varepsilon_h^\partial}$ and $\norm {y-y_h}_{\mathcal T_h}$}

Next, we bound the error between the solutions of the auxiliary problem and the HDG problem \eqref{HDG_full_discrete}.  We use these error bounds and the error bounds in Lemma \ref{lemma:step3_conv_rates}, Lemma \ref{lemma:step5_main_lemma}, and Lemma \ref{lemma:step5_conv_rates} to obtain the main result.

For the remaining steps, we denote
\begin{equation*}
\begin{split}
\zeta_{\bm q} &=\bm q_h(u)-\bm q_h,\quad\zeta_{y} = y_h(u)-y_h,\quad\zeta_{\widehat y} = \widehat y_h(u)-\widehat y_h,\\
\zeta_{\bm p} &=\bm p_h(u)-\bm p_h,\quad\zeta_{z} = z_h(u)-z_h,\quad\zeta_{\widehat z} = \widehat z_h(u)-\widehat z_h,
\end{split}
\end{equation*}
where $ \widehat y_h = \widehat y_h^o $ on $ \varepsilon_h^o $, $ \widehat y_h = u_h $ on $ \varepsilon_h^\partial $, $ \widehat z_h = \widehat z_h^o $ on $ \varepsilon_h^o $, and $ \widehat z_h = 0 $ on $ \varepsilon_h^\partial $.  This gives $ \zeta_{\widehat z} = 0 $ on $ \varepsilon_h^\partial $.

Subtracting the auxiliary problem and the HDG problem gives the following error equations
%
\begin{subequations}\label{eq_yh}
	\begin{align}
	\mathscr B(\zeta_{\bm q},\zeta_y,\zeta_{\widehat y};\bm r_1, w_1,\mu_1) &= -\langle P_Mu-u_h, \bm{r_1}\cdot \bm{n} - h^{-1} w_1 \rangle_{{\varepsilon_h^{\partial}}}\label{eq_yh_yhu},\\
	\mathscr B(\zeta_{\bm p},\zeta_z,\zeta_{\widehat z};\bm r_2, w_2,\mu_2) &= (\zeta_y, w_2)_{\mathcal T_h}\label{eq_zh_zhu},
	\end{align}
\end{subequations}
for all $\left(\bm{r}_1, \bm{r}_2, w_1, w_2, \mu_1, \mu_2\right)\in \bm V_h\times\bm V_h\times W_h\times W_h\times M_h(o)\times M_h(o) $.

\begin{lemma}
	We have
	\begin{align*}
	\hspace{2em}&\hspace{-2em} \norm{u-u_h}_{\varepsilon_h^\partial}^2  + \gamma^{-1}\norm {\zeta_y}_{\mathcal T_h}^2 \\
	&= \langle u+\gamma^{-1}\bm p_h(u)\cdot\bm n +\gamma^{-1}h^{-1} P_M z_h(u),u-u_h\rangle_{\varepsilon_h^\partial}\\
	& \quad - \langle u_h+\gamma^{-1}\bm p_h\cdot\bm n +\gamma^{-1}h^{-1} P_M z_h,u-u_h\rangle_{\varepsilon_h^\partial}.
	\end{align*}
\end{lemma}
\begin{proof}
	First, we have
	\begin{align*}
	&\langle u+\gamma^{-1}\bm p_h(u)\cdot\bm n +\gamma^{-1}h^{-1} P_Mz_h(u),u-u_h\rangle_{\varepsilon_h^\partial}\\
	& \quad - \langle u_h+\gamma^{-1}\bm p_h\cdot\bm n +\gamma^{-1}h^{-1} P_Mz_h,u-u_h\rangle_{\varepsilon_h^\partial}\\
	&= \norm{u-u_h}_{\varepsilon_h^\partial}^2 + \gamma^{-1} \langle \zeta_{\bm p}\cdot\bm n+h^{-1} P_M \zeta_z,u-u_h\rangle_{\varepsilon_h^\partial}.
	\end{align*}
	As in the proof of Lemma \ref{eq_B},  it can be shown that
	\begin{align*}
	\mathscr B(\zeta_{\bm q},\zeta_y,\zeta_{\widehat y};\zeta_{\bm p}, -\zeta_{z},-\zeta_{\widehat{ z}}) + \mathscr B(\zeta_{\bm p},\zeta_z,\zeta_{\widehat z}; -\zeta_{\bm q}, \zeta_{y},\zeta_{\widehat{y}}) = 0.
	\end{align*}
	One the other hand, we have
	\begin{align*}
	\mathscr B(\zeta_{\bm q},\zeta_y,\zeta_{\widehat y};\zeta_{\bm p}, -\zeta_{z},-\zeta_{\widehat{ z}}) &+ \mathscr B(\zeta_{\bm p},\zeta_z,\zeta_{\widehat z}; -\zeta_{\bm q}, \zeta_{y},\zeta_{\widehat{y}})\\
	&=  (\zeta_{ y},\zeta_{ y})_{\mathcal{T}_h} - \langle P_Mu-u_h, \zeta_{\bm p}\cdot \bm{n} + h^{-1}\zeta_z \rangle_{{\varepsilon_h^{\partial}}}\\
	&= (\zeta_{ y},\zeta_{ y})_{\mathcal{T}_h} -\langle u-u_h, \zeta_{\bm p}\cdot \bm{n} + h^{-1}P_M\zeta_z \rangle_{{\varepsilon_h^{\partial}}}.
	\end{align*}
	Comparing the above two equalities gives
	\begin{align*}
	(\zeta_{ y},\zeta_{ y})_{\mathcal{T}_h} = \langle u-u_h, \zeta_{\bm p}\cdot \bm{n} + h^{-1}P_M\zeta_z \rangle_{{\varepsilon_h^{\partial}}}.
	\end{align*}
\end{proof}

\begin{thm}
	We have
	\begin{align*}
	\norm{u-u_h}_{\varepsilon_h^\partial}&\lesssim h^{s_{\bm p}-\frac 1 2}\norm{\bm p}_{s_{\bm p},\Omega} +  h^{s_{z}-\frac 3 2}\norm{z}_{s_{z},\Omega} + h^{s_{\bm q}+\frac 1 2}\norm{\bm q}_{s_{\bm q},\Omega} + h^{s_{y}-\frac 12 }\norm{y}_{s_{y},\Omega},\\
	\norm {y-y_h}_{\mathcal T_h} &\lesssim h^{s_{\bm p}-\frac 1 2}\norm{\bm p}_{s_{\bm p},\Omega} +  h^{s_{z}-\frac 3 2}\norm{z}_{s_{z},\Omega} + h^{s_{\bm q}+\frac 1 2}\norm{\bm q}_{s_{\bm q},\Omega} + h^{s_{y}-\frac 12 }\norm{y}_{s_{y},\Omega}.
	\end{align*}
\end{thm}

\begin{proof}
	%
	Since $u+\gamma^{-1}\bm p \cdot\bm n=0$ on $\varepsilon_h^{\partial}$ and $u_h+\gamma^{-1}\bm p_h\cdot\bm n +\gamma^{-1}h^{-1}P_M z_h=0$ on $\varepsilon_h^{\partial}$ we have
	\begin{align*}
	\norm{u-u_h}_{\varepsilon_h^\partial}^2  + \gamma^{-1}\norm {\zeta_y}_{\mathcal T_h}^2 &= \langle u+\gamma^{-1}\bm p_h(u)\cdot\bm n +\gamma^{-1}h^{-1} P_Mz_h(u),u-u_h\rangle_{\varepsilon_h^\partial}\\
	&=\langle \gamma^{-1}(\bm p_h(u)-\bm p)\cdot\bm n + \gamma^{-1}h^{-1} P_Mz_h(u) ,u-u_h\rangle_{\varepsilon_h^\partial}\\
	&\lesssim (\norm {\bm p_h(u)-\bm p}_{\partial \mathcal T_h} +h^{-1}\|P_Mz_h(u) \|_{\varepsilon_h^\partial})\norm{u-u_h}_{\varepsilon_h^\partial}.
	\end{align*}
	Next, since $\widehat z_h(u) = z=0$ on $\varepsilon_h^{\partial}$ we have
	\begin{align*}
	\norm {\bm p_h(u)-\bm p}_{\partial \mathcal T_h} &\le \norm {\bm p_h(u)-\bm{\Pi}\bm p}_{\partial \mathcal T_h} +\norm {\bm{\Pi}\bm p - \bm p}_{\partial \mathcal T_h}\\
	&\lesssim  h^{-\frac 1 2}\norm {\bm p_h(u)-\bm{\Pi}\bm p}_{\mathcal T_h} +h^{s_{\bm p}-\frac 12 } \norm{\bm p}_{s^{\bm p},\Omega}\\
	& \lesssim h^{s_{\bm p}-\frac 1 2}\norm{\bm p}_{s_{\bm p},\Omega} +  h^{s_{z}-\frac 3 2}\norm{z}_{s_{z},\Omega} + h^{s_{\bm q}+\frac 1 2}\norm{\bm q}_{s_{\bm q},\Omega}\\
	& \quad + h^{s_{y}-\frac 12 }\norm{y}_{s_{y},\Omega},\\
	\|P_Mz_h(u)\|_{\varepsilon_h^\partial}&=\|P_Mz_h(u) -P_M\Pi z+P_M\Pi z -P_Mz +P_M z- \widehat z_h(u) \|_{\varepsilon_h^\partial} \\
	& \le (\|P_M\varepsilon_h^z -\varepsilon_h^{\widehat z}\|_{\varepsilon_h^\partial} + \norm {\Pi z- z}_{\varepsilon_h^\partial})\\
	& \le (\|P_M\varepsilon_h^z -\varepsilon_h^{\widehat z}\|_{\partial\mathcal T_h} + \norm {\Pi z- z}_{\partial\mathcal T_h} ).
	\end{align*}
	Lemma \ref{lemma:step5_main_lemma} and properties of the $ L^2 $ projection gives
	\begin{align*}
	\norm{u-u_h}_{\varepsilon_h^\partial}\lesssim h^{s_{\bm p}-\frac 1 2}\norm{\bm p}_{s_{\bm p},\Omega} +  h^{s_{z}-\frac 3 2}\norm{z}_{s_{z},\Omega} + h^{s_{\bm q}+\frac 1 2}\norm{\bm q}_{s_{\bm q},\Omega} + h^{s_{y}-\frac 12 }\norm{y}_{s_{y},\Omega}.
	\end{align*}
	Moreover, we have
	\begin{align*}
	\norm {\zeta_y}_{\mathcal T_h} \lesssim h^{s_{\bm p}-\frac 1 2}\norm{\bm p}_{s_{\bm p},\Omega} +  h^{s_{z}-\frac 3 2}\norm{z}_{s_{z},\Omega} + h^{s_{\bm q}+\frac 1 2}\norm{\bm q}_{s_{\bm q},\Omega} + h^{s_{y}-\frac 12 }\norm{y}_{s_{y},\Omega}.
	\end{align*}
	Then, by the triangle inequality and Lemma \ref{lemma:step3_conv_rates} we obtain
	\begin{align*}
	\norm {y - y_h}_{\mathcal T_h} \lesssim h^{s_{\bm p}-\frac 1 2}\norm{\bm p}_{s_{\bm p},\Omega} +  h^{s_{z}-\frac 3 2}\norm{z}_{s_{z},\Omega} + h^{s_{\bm q}+\frac 1 2}\norm{\bm q}_{s_{\bm q},\Omega} + h^{s_{y}-\frac 12 }\norm{y}_{s_{y},\Omega}.
	\end{align*}
\end{proof}

%
\subsubsection{Step 7: Estimates for $\|\boldmath q - \boldmath q_h\|_{\mathcal T_h}$, $\|\boldmath p-\boldmath p_h\|_{\mathcal T_h}$ and $\|z-z_h\|_{\mathcal T_h}$}
\begin{lemma}
	We have
	\begin{align*}
	\norm {\zeta_{\bm q}}_{\mathcal T_h}  &\lesssim h^{s_{\bm p}-1}\norm{\bm p}_{s_{\bm p},\Omega} +  h^{s_{z}-2}\norm{z}_{s_{z},\Omega} + h^{s_{\bm q}}\norm{\bm q}_{s_{\bm q},\Omega} + h^{s_{y}-1 }\norm{y}_{s_{y},\Omega},\\
	\norm {\zeta_{\bm p}}_{\mathcal T_h}  &\lesssim h^{s_{\bm p}-\frac 1 2}\norm{\bm p}_{s_{\bm p},\Omega} +  h^{s_{z}-\frac 3 2}\norm{z}_{s_{z},\Omega} + h^{s_{\bm q}+\frac 1 2}\norm{\bm q}_{s_{\bm q},\Omega} + h^{s_{y}-\frac 12 }\norm{y}_{s_{y},\Omega},\\
	\|\zeta_z\|_{\mathcal T_h} & \lesssim  h^{s_{\bm p}-\frac 1 2}\norm{\bm p}_{s_{\bm p},\Omega} +  h^{s_{z}-\frac 3 2}\norm{z}_{s_{z},\Omega} + h^{s_{\bm q}+\frac 1 2}\norm{\bm q}_{s_{\bm q},\Omega} + h^{s_{y}-\frac 12 }\norm{y}_{s_{y},\Omega}.
	\end{align*}
\end{lemma}
\begin{proof}
	By Lemma \ref{energy_norm} and the error equation \eqref{eq_yh_yhu}, we have
	\begin{align*}
	\hspace{2em}&\hspace{-2em}\mathscr B(\zeta_{\bm q},\zeta_y,\zeta_{\widehat y};\zeta_{\bm q},\zeta_y,\zeta_{\widehat y}) \\
	&=(\zeta_{\bm q}, \zeta_{\bm q})_{{\mathcal{T}_h}}+\langle h^{-1} (P_M\zeta_y-\zeta_{\widehat y}) , P_M\zeta_y-\zeta_{\widehat y}\rangle_{\partial{{\mathcal{T}_h}}\backslash\varepsilon_h^{\partial}} + \langle h^{-1} P_M\zeta_y, P_M\zeta_y\rangle_{\varepsilon_h^{\partial}}\\
	&= -\langle P_M u-u_h, \zeta_{\bm q}\cdot \bm{n} - h^{-1} \zeta_y \rangle_{{\varepsilon_h^{\partial}}}  = -\langle u-u_h, \zeta_{\bm q}\cdot \bm{n} - h^{-1} P_M\zeta_y \rangle_{{\varepsilon_h^{\partial}}}\\
	&\lesssim \norm {u-u_h}_{\varepsilon_h^{\partial}} (\norm {\bm\zeta_{\bm q}}_{\varepsilon_h^{\partial}} + h^{-1} \norm {P_M\zeta_{y}}_{\varepsilon_h^{\partial}} )\\
	&\lesssim h^{-\frac 1 2}\norm {u-u_h}_{\varepsilon_h^{\partial}} (\norm {\bm\zeta_{\bm q}}_{\mathcal T_h} + h^{-\frac 1 2} \norm {P_M\zeta_{y}}_{\varepsilon_h^{\partial}}),
	\end{align*}
	which gives
	\begin{align*}
	\norm {\zeta_{\bm q}}_{\mathcal T_h} &\lesssim h^{-\frac 1 2}\norm {u-u_h}_{\varepsilon_h^{\partial}} \\
	&\lesssim h^{s_{\bm p}-1}\norm{\bm p}_{s_{\bm p},\Omega} +  h^{s_{z}-2}\norm{z}_{s_{z},\Omega} + h^{s_{\bm q}}\norm{\bm q}_{s_{\bm q},\Omega} + h^{s_{y}-1 }\norm{y}_{s_{y},\Omega}.
	\end{align*}
	
	Next, we estimate $\zeta_{\bm p}$.  By Lemma \ref{energy_norm}, the error equation \eqref{eq_zh_zhu}, and since $\zeta_{ \widehat z} = 0$ on $\varepsilon_h^{\partial}$, we have
	\begin{align*}
	\hspace{2em}&\hspace{-2em}\mathscr B(\zeta_{\bm p},\zeta_z,\zeta_{\widehat z};\zeta_{\bm p},\zeta_z,\zeta_{\widehat z})\\
	&=(\zeta_{\bm p}, \zeta_{\bm p})_{{\mathcal{T}_h}}+\langle h^{-1} (P_M\zeta_z-\zeta_{\widehat z}) , P_M\zeta_z-\zeta_{\widehat z}\rangle_{\partial{{\mathcal{T}_h}}\backslash\varepsilon_h^{\partial}} + \langle h^{-1} P_M\zeta_z, P_M\zeta_z\rangle_{\varepsilon_h^{\partial}}\\
	&=(\zeta_{\bm p}, \zeta_{\bm p})_{{\mathcal{T}_h}}+\langle h^{-1} (P_M\zeta_z-\zeta_{\widehat z}) , P_M\zeta_z-\zeta_{\widehat z}\rangle_{\partial{{\mathcal{T}_h}}}\\
	&=(\zeta_y,\zeta_z)_{\mathcal T_h}\\
	&\le \norm{\zeta_y}_{\mathcal T_h} \norm{\zeta_z}_{\mathcal T_h}\\
	&\lesssim \norm{\zeta_y}_{\mathcal T_h} (\|\nabla \zeta_z\|_{\mathcal T_h} + h^{-\frac 1 2} \|\zeta_z - \zeta_{\widehat z}\|_{\partial\mathcal T_h}) \\
	& \lesssim \norm{\zeta_y}_{\mathcal T_h}
	(\|\zeta_{\bm p}\|_{\mathcal T_h}+h^{-\frac1 2}\|P_M\zeta_z-\zeta_{\widehat z}\|_{\partial\mathcal T_h}),
	\end{align*}
	where we used the discrete Poincar{\'e} inequality in Lemma \ref{lemma:discr_Poincare_ineq} and also Lemma \ref{nabla_ine}.  This implies
	\begin{align*}
	\hspace{2em}&\hspace{-2em} \norm {\zeta_{\bm p}}_{\mathcal T_h} +h^{-\frac1 2}\|P_M\zeta_z-\zeta_{\widehat z}\|_{\partial\mathcal T_h}\\
	& \quad \lesssim h^{s_{\bm p}-\frac 1 2}\norm{\bm p}_{s_{\bm p},\Omega} +  h^{s_{z}-\frac 3 2}\norm{z}_{s_{z},\Omega} + h^{s_{\bm q}+\frac 1 2}\norm{\bm q}_{s_{\bm q},\Omega} + h^{s_{y}-\frac 12 }\norm{y}_{s_{y},\Omega}.
	\end{align*}
	
	The discrete Poincar{\'e} inequality in Lemma \ref{lemma:discr_Poincare_ineq} also gives
	\begin{align*}
	\|\zeta_z\|_{\mathcal T_h} & \lesssim \|\nabla \zeta_z\|_{\mathcal T_h} + h^{-\frac 1 2} \|\zeta_z - \zeta_{\widehat z}\|_{\partial\mathcal T_h}\\
	&\lesssim \|\zeta_{\bm p}\|_{\mathcal T_h}+h^{-\frac1 2}\|P_M\zeta_z-\zeta_{\widehat z}\|_{\partial\mathcal T_h}\\
	&\lesssim h^{s_{\bm p}-\frac 1 2}\norm{\bm p}_{s_{\bm p},\Omega} +  h^{s_{z}-\frac 3 2}\norm{z}_{s_{z},\Omega} + h^{s_{\bm q}+\frac 1 2}\norm{\bm q}_{s_{\bm q},\Omega} + h^{s_{y}-\frac 12 }\norm{y}_{s_{y},\Omega}.
	\end{align*}
\end{proof}

The above lemma along with the triangle inequality, Lemma \ref{lemma:step3_conv_rates}, and Lemma \ref{lemma:step5_conv_rates} complete the proof of the main result:
\begin{theorem}
	We have
	\begin{align*}
	\norm {\bm q - \bm q_h}_{\mathcal T_h} &\lesssim h^{s_{\bm p}-1}\norm{\bm p}_{s_{\bm p},\Omega} +  h^{s_{z}-2}\norm{z}_{s_{z},\Omega} + h^{s_{\bm q}}\norm{\bm q}_{s_{\bm q},\Omega} + h^{s_{y}-1 }\norm{y}_{s_{y},\Omega},\\
	\norm {\bm p - \bm p_h}_{\mathcal T_h}   &\lesssim h^{s_{\bm p}-\frac 1 2}\norm{\bm p}_{s_{\bm p},\Omega} +  h^{s_{z}-\frac 3 2}\norm{z}_{s_{z},\Omega} + h^{s_{\bm q}+\frac 1 2}\norm{\bm q}_{s_{\bm q},\Omega} + h^{s_{y}-\frac 12 }\norm{y}_{s_{y},\Omega},\\
	\norm {z - z_h}_{\mathcal T_h}   & \lesssim  h^{s_{\bm p}-\frac 1 2}\norm{\bm p}_{s_{\bm p},\Omega} +  h^{s_{z}-\frac 3 2}\norm{z}_{s_{z},\Omega} + h^{s_{\bm q}+\frac 1 2}\norm{\bm q}_{s_{\bm q},\Omega} + h^{s_{y}-\frac 12 }\norm{y}_{s_{y},\Omega}.
	\end{align*}
\end{theorem}

\section{Numerical Experiments}
\label{sec:numerics}

For our numerical experiments, we test problems similar to the examples considered in \cite{MR2806572}; see also \cite{MR3317816,MR2272157,MR3070527}.  We chose $ k = 1 $ for all computations; i.e., quadratic polynomials are used for the scalar variables, and linear polynomials are used for the flux variables and the boundary trace variables.

We begin with a 2D example on a square domain $\Omega = [0,1/4]\times [0,1/4] \subset \mathbb{R}^2$.  The largest interior angle is $ \omega = {\pi}/{2}$, and so $ r_\Omega = 3 $ and $p_{\Omega} = \infty$.  The data is chosen as
\begin{align*}
f=0, \ \  y_d = (x^2+y^2)^{s}  \ \ \ \mbox{and} \  \ \ \gamma = 1,
\end{align*}
where $s=10^{-5}$.  Then $y_d\in H^1(\Omega) \cap L^{\infty}(\Omega)$, and Corollary \ref{main_reslut} in Section \ref{sec:analysis} gives the convergence rates
\begin{align*}
&\norm{y-{y}_h}_{0,\Omega}=O( h^{3/2-\varepsilon} ),\quad  \;\norm{z-{z}_h}_{0,\Omega}=O( h^{3/2-\varepsilon} ),\\
&\norm{\bm{q}-\bm{q}_h}_{0,\Omega} = O( h^{1-\varepsilon} ),\quad  \quad\;\; \norm{\bm{p}-\bm{p}_h}_{0,\Omega} = O( h^{3/2-\varepsilon} ),
\end{align*}
and
\begin{align*}
&\norm{u-{u}_h}_{0,\Gamma} = O( h^{3/2-\varepsilon}).
\end{align*}

Since we do not have an explicit expression for the exact solution, we solved the problem numerically for a triangulation with 262144 elements, i.e., $h = 2^{-12}\sqrt{2}$ and compared this reference solution against other solutions computed on meshes with larger $h$.  The numerical results are shown in Table \ref{table_1}.  The convergence rates observed for $\norm{\bm{q}-\bm{q}_h}_{0,\Omega}$ and $\norm{{u}-{u}_h}_{0,\Gamma}$ are in agreement with our theoretical results, while the convergence rates for $\norm{\bm{p}-\bm{p}_h}_{0,\Omega}$,  $\norm{{y}-{y}_h}_{0,\Omega}$, and $\norm{{z}-{z}_h}_{0,\Omega}$  are higher than our theoretical results.  A similar phenomena can be observed in  \cite{MR3070527,MR3317816,MR2806572}.
\begin{table}[!hbp]
	\begin{center}
		\begin{tabular}{|c|c|c|c|c|c|}
			\hline
			$h/\sqrt{2}$ &$2^{-4}$& 1/$2^{-5}$&$2^{-6}$ &$2^{-7}$ & $2^{-8}$ \\
			\hline
			$\norm{\bm{q}-\bm{q}_h}_{0,\Omega}$&4.1343e-02& 2.1025e-02&1.0677e-02  & 5.3865e-03 & 2.6959e-03 \\
			\hline
			order&-& 0.9756& 0.9776  &0.9871& 0.9986 \\
			\hline
			$\norm{\bm{p}-\bm{p}_h}_{0,\Omega}$&1.3463e-03 & 3.8638e-04&1.0849e-04 &2.9862e-05 & 8.0969e-06 \\
			\hline
			order&-&  1.8009&1.8325  &1.8612  & 1.8828 \\
			\hline
			$\norm{{y}-{y}_h}_{0,\Omega}$&5.4609e-04& 1.3647e-04 &3.4763e-05 &8.8037e-06  & 2.2236e-06\\
			\hline
			order&-& 2.0005&1.9730&1.9814 & 1.9852 \\
			\hline
			$\norm{{z}-{z}_h}_{0,\Omega}$&1.9671e-05& 2.6887e-06&3.7026e-07  &5.0372e-08 & 6.7767e-09 \\
			\hline
			order&-& 2.8711&2.8603 &2.8778& 2.8940 \\
			\hline
			$\norm{{u}-{u}_h}_{0,\Gamma}$&7.3053e-03& 2.6902e-03&9.7764e-04 &3.5178e-04  & 1.2569e-04 \\
			\hline
			order&-&  1.4412& 1.4603  &1.4746& 1.4849 \\
			\hline
		\end{tabular}
	\end{center}
	\caption{Error of control $u$, state $y$, adjoint state $z$, and their fluxes $\bm q$ and $\bm p$.}\label{table_1}
\end{table}



For illustration, we plot the state $y$, adjoint state $z$, and their fluxes $\bm q$ and $\bm p$.  The 2D regularity result in Section \ref{sec:background} indicate that the primary flux $ \bm q $ can have low regularity.  In this example, it does indeed appear that $ \bm q $ has singularities at the corners of the domain.  These figures can be compared to similar plots in \cite{MR3317816,MR2272157}.
\begin{figure}[hbt]
	\centerline{
		\hbox{\includegraphics[width=\linewidth]{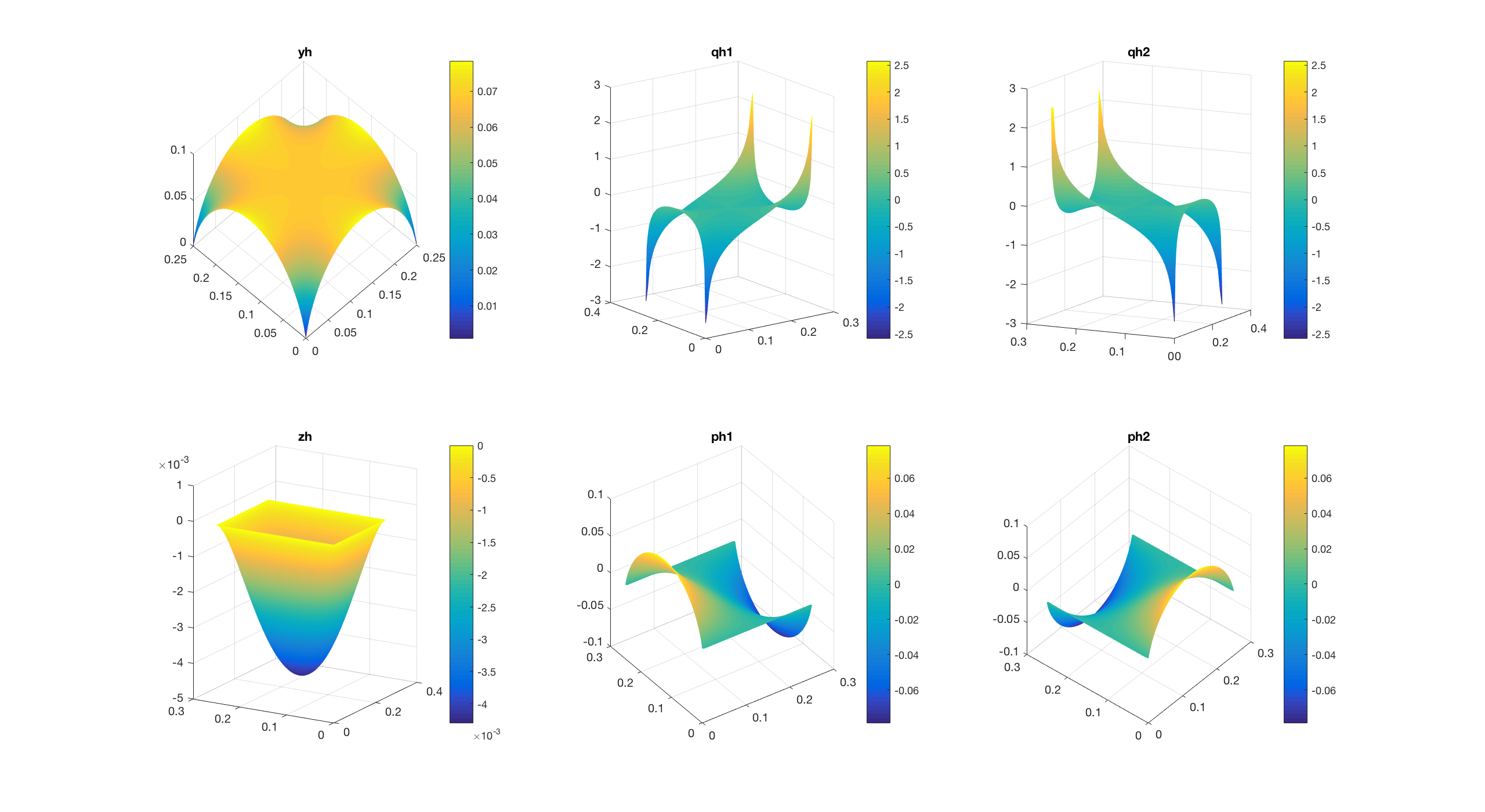}}
	}
	\caption{The primary state $y_h$, the primary flux $\bm q_h$, the dual state $z_h$, and the dual flux $\bm p_h$ for the 2D example}
	\centering
\end{figure}

%
%
%
%
%
%
\begin{figure}[hbt]
	\centerline{
		\hbox{\includegraphics[height=3in]{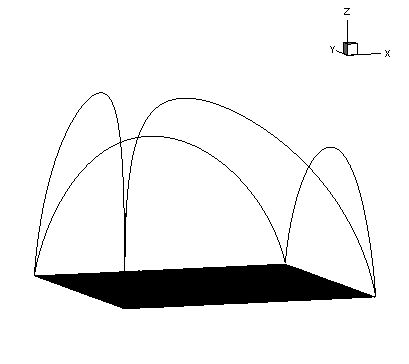}}}
	\caption{The optimal control $u_h$ for the 2D example}
	\centering
\end{figure}


Next, we consider a 3D extension of the 2D example above. The domain is a cube $\Omega=[0,1/32]\times [0,1/32]\times [0,1/32]$, and the data is chosen as
\begin{align*}
f=0, \ \  y_d = (x^2+y^2+z^2)^{s} \quad \mbox{and} \quad  \gamma = 1,
\end{align*}
where $s= -1/4+10^{-5}$, so that $y_d\in H^1(\Omega)$. In this case, we did not attempt to determine the regularity of the control and other variables; we simply present the numerical results here.

As in the 2D example above, we do not have an explicit expression for the exact solution.  Therefore, we solved the problem numerically for a triangulation with 196608 tetrahedrons, i.e., $h = 2^{-12}\sqrt{3}$ and compared this reference solution against other solutions computed on meshes with larger $h$.  The numerical results are shown in Table \ref{table_2}. The observed convergence rates for all variables are similar to the results for the 2D example above.

\begin{table}[!hbp]
	\begin{center}
		\begin{tabular}{|c|c|c|c|c|c|}
			\hline
			$h/\sqrt{3}$ &$2^{-6}$& $2^{-7}$&$2^{-8}$ &$2^{-9}$ \\
			\hline
			$\norm{\bm{q}-\bm{q}_h}_{0,\Omega}$&9.2640e-03& 5.2580e-03 & 2.7462e-03  & 1.2475e-03 \\
			\hline
			order&-& 0.81712&  0.93706  &1.1384 \\
			\hline
			$\norm{\bm{p}-\bm{p}_h}_{0,\Omega}$&3.5425e-05& 1.2283e-05&3.8463e-06 &1.1022e-06  \\
			\hline
			order&-&  1.5281&1.6751  &1.8032  \\
			\hline
			$\norm{{y}-{y}_h}_{0,\Omega}$&1.6040e-05& 4.5070e-06 &1.2191e-06 &2.9781e-07  \\
			\hline
			order&-& 1.8314&1.8864& 2.0333  \\
			\hline
			$\norm{{z}-{z}_h}_{0,\Omega}$&7.8545e-08& 1.3058e-08 &2.0042e-09  &2.8775e-10  \\
			\hline
			order&-&2.5886 &2.7039 &2.8001\\
			\hline
			$\norm{{u}-{u}_h}_{0,\Gamma}$&4.5932e-04&  1.8934e-04&7.1955e-05 &2.4123e-05   \\
			\hline
			order&-&  1.2785& 1.3958 & 1.5767 \\
			\hline
		\end{tabular}
	\end{center}
	\caption{Error of control $u$, state $y$, adjoint state $z$, and their fluxes $\bm q$ and $\bm p$.}\label{table_2}
\end{table}
\section{Conclusions}

We proposed an HDG method to approximate the solution of an optimal Dirichlet boundary control problems for the Poisson equation.  We obtained a superlinear rate of convergence for the control in 2D under certain assumptions on the domain and the target state $ y_d $.  Numerical experiments confirmed our theoretical results.

Our results indicate HDG methods have potential for solving more complex Dirichlet boundary control problems.  We plan to investigate HDG methods for Dirichlet boundary control of other PDEs, including convection dominated diffusion problems and fluid flows.  These problems may involve solutions with large gradients or shocks, and it is natural to consider HDG methods for such problems.  


\section*{Acknowledgements} The authors thank Bernardo Cockburn for many valuable conversations.

\appendix
\section{Local Solver}

By simple algebraic operations in equation \eqref{system_equation2}, we obtain the following formulas for the matrices $ G_1 $, $ G_2 $, $ H_1 $, and $ H_2 $ in \eqref{local_solver}:
\begin{align*}
G_1 &= B_1^{-1}B_2(B_4+B_2^TB_1^{-1}B_2)^{-1}(B_5+B_2^TB_1^{-1}B_3)-B_1^{-1}B_3,\\
G_2 &= -(B_4+B_2^TB_1^{-1}B_2)^{-1}(B_5+B_2^TB_1^{-1}B_3),\\
H_1 &= -B_1^{-1}B_2(B_4+B_2^TB_1^{-1}B_2)^{-1},\\
H_2 &= (B_4+B_2^TB_1^{-1}B_2)^{-1}.
\end{align*}
In general, this process is impractical; however, for the HDG method described in this work, these matrices can be easily computed.  This is one of the advantages of the HDG method.  We briefly describe this process below.


Since the spaces $ \bm{V}_h $ and $ W_h $ consist of discontinuous polynomials, some of the system matrices are block diagonal and each block is small and symmetric positive definite.  Let us call a matrix of this form a SSPD block diagonal matrix.  The inverse of a SSPD block diagonal matrix is another SSPD block diagonal matrix, and the inverse can be easily constructed by computing the inverse of each small block.  Furthermore, the inverse of each small block can be computed independently; and therefore computing the inverse can be easily done in parallel.

It can be checked that $ B_1 $ is a SSPD block diagonal matrix, and therefore $ B_1^{-1} $ is easily computed and is also a SSPD block diagonal matrix.  Therefore, the the matrices $ G_1 $, $ G_2 $, $ H_1 $, and $ H_2 $ are easily computed if $ B_4 + B_2^T B_1^{-1} B_2 $ is also easily inverted.  We show below that this is the case.





First, it can be checked that $ B_2 $ is block diagonal with small blocks, but the blocks are not symmetric or definite.  This implies $B_2^T B_1^{-1} B_2$ is block diagonal with small nonnegative definite blocks.  Next, $B_4 = \begin{bmatrix}
A_5 & 0\\
-A_4 & A_5
\end{bmatrix}$, where $A_4$ and $A_5$ are both SSPD block diagonal.  Due to the structure of $ B_1 $ and $ B_2 $, the matrix $B_2^TB_1^{-1}B_2 + B_4$ has the form
$\begin{bmatrix}
C_1 & 0\\
-A_4 & C_2
\end{bmatrix},
$
where $ C_1 $ and $ C_2 $ are SSPD block diagonal.  The inverse can be easily computed using the formula
$$\begin{bmatrix}
C_1 & 0\\
-A_4 & C_2
\end{bmatrix}^{-1} = \begin{bmatrix}
C_1^{-1} & 0\\
C_2^{-1} A_4 C_1^{-1} & C_2^{-1}
\end{bmatrix}.
$$
Furthermore, $ C_1^{-1} $, $C_2^{-1}$ and $ C_2^{-1} A_4 C_1^{-1} $ are both SSPD block diagonal.



\bibliographystyle{plain}
\bibliography{yangwen_ref_papers,yangwen_ref_books}
\end{document}